\documentclass[12pt,reqno]{amsart}

\usepackage[final,
color]{showkeys}
\definecolor{refkey}{gray}{.85}
\definecolor{labelkey}{gray}{.85}

\usepackage{AKstyle}

\numberwithin{equation}{section}

\usepackage{caption}
\usepackage[labelfont=rm]{subcaption}

\usepackage[pagebackref=true, colorlinks]{hyperref}

\hypersetup{pdffitwindow=true,linkcolor=blue,citecolor=blue,urlcolor=blue,menucolor=blue}

\usepackage{comment}

\begin{document}

\author{Dubi Kelmer}
\thanks{Kelmer is partially supported by NSF grant DMS-1237412.}
\email{kelmer@bc.edu}
\address{Boston College, Boston, MA}
\author{Alex Kontorovich}
\thanks{Kontorovich is partially supported by
an NSF CAREER grant DMS-1254788, an Alfred P. Sloan Research Fellowship, and a Yale Junior Faculty Fellowship.}
\email{alex.kontorovich@yale.edu}
\address{Yale University, New Haven, CT}

\title
{On the Pair Correlation Density for Hyperbolic Angles}

\begin{abstract}
Let $\Gamma< \PSL_2(\R)$ be a lattice and $\omega\in \bH$ a point in the upper half plane. We prove the existence and give an explicit formula for the %the
pair correlation density function for the set of angles between geodesic rays of the lattice $\G \omega$ intersected with increasingly large balls centered at $\omega$, thus proving
a conjecture of Boca-Popa-Zaharescu.
%on the density function for the pair correlation of hyperbolic angles in a growing orbit of a lattice in $\PSL_{2}(\R)$.
\end{abstract}
\date{\today}
\maketitle
\tableofcontents

\section{Introduction}\label{sec:intro}

Let $\G< G:=\PSL_2(\R)$ be a lattice (i.e., a co-finite Fuchsian group) acting by isometries on the upper half plane $\bH$ via fractional linear transformations.
Given a fixed base point $
%(\gw,\nu)\in T^{1}\bH
\gw\in\bH
$% in the unit tangent bundle
, consider the set of directions of  % set of angles between $\nu$ and
 geodesic rays
connecting $\omega$ to %the finitely many
points $\g \omega$ lying in a
%n increasingly
%large
growing
hyperbolic norm ball in $\bH$ (see Figure \ref{fig:1a} for an illustration). %for $\gw=i$ and $\G$ %is
%the uniform arithmetic lattice corresponding to the
%spin cover of the special orthogonal group preserving the ($\Q$-anisotropic, ternary, indefinite) form $x^{2}+y^{2}-3z^{2}$; more explicitly,
%% the norm one integer elements of the quaternion division algebra with norm
%%$$
%%\N(a+bI+cJ+dK)=a^{2}-3b^{2}-3c^{2}+d^{2},
%%$$
%%via the map
%\be\label{eq:GamUniform}
%%a+bI+cJ+dK\mapsto
%\G=
%\left\{
%%\g=
%\mattwo
%{a+\sqrt3b}{-d-\sqrt3c}{d-\sqrt3c}{a-\sqrt3b}
%:
%%\det\g=1,\quad
%a,b,c,d\in\Z,\ %\text{ and }
%a^{2}-3b^{2}-3c^{2}+d^{2}=1
%\right\}
%.
%\ee
\begin{figure}
\includegraphics[width=5in]{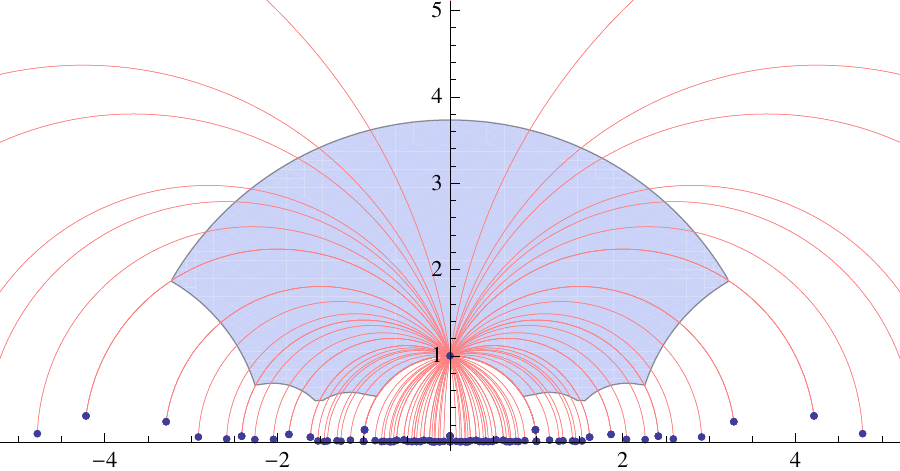}
\caption{%
Geodesic rays connecting $\gw=i$ to points $\g\gw$ for the uniform lattice $\G$ corresponding to the
spin cover of the special orthogonal group preserving the ($\Q$-anisotropic, ternary, indefinite) form $x^{2}+y^{2}-3z^{2}$.
A fundamental domain for $\G$ is shaded.
}
\label{fig:1a}
\end{figure}

It is classical
that these angles become equidistributed in $\R/(2\pi\Z)$%, with effective rates of equidistribution known
; see, e.g., \cite{%Delsarte1942, Huber1956, Selberg1956, Patterson1975B, LaxPhillips1982,
Boca2007,Nicholls1983,GoodBook,RisagerTruelsen2010}.
Going beyond equidistribution, the finer structure of a set of real numbers can be measured by (among other statistics) its
pair correlation, measuring the
distribution of spacings at distances of
mean order. Such spacing statistics have been studied by many authors for many naturally occurring sequences arising in mathematical physics, analysis, and number theory, both experimentally and theoretically. In particular, for the Euclidian analogue of this problem, the pair correlation, as well as other spacing statistics, were studied in \cite{BocaCobeliZaharescu2000,BocaZaharescu2005,BocaZaharescu2006,MarklofStrombergsson2010,ElBazMarklofVinogradov2013}.
It is thus very surprising that the question of spacing statistics for the well-studied set of geodesic ray angles % described above
in the hyperbolic plane
was considered for the first time only recently  in \cite{BocaPasolPopaZaharescu2012,BocaPopaZaharescu2013}.
To state the %ir
results, % and ours
 we introduce some notation.

Fix a point $(\gw,\nu)\in T^{1}\bH$ in the unit tangent bundle. For any $g\in G$ with $g\gw\neq\gw$, let
% (after conjugation, we may assume $\gw=i$).
%For $g\in G:=\PSL_2(\R)$, let
\be\label{eq:angle}
\gt_{g}=\gt_{g}(\gw,\nu) \ \in \ \R/(2\pi\Z)
\ee
be the signed angle between the vector
%this tangent vector
$\nu$
and the tangent vector at $\gw$ of the directed geodesic connecting $\gw\to g\gw$, see Figure \ref{fig:1}. When $g\gw=\gw$, set $\gt_{g}=[0]$.
%Let
%$\G<G$ %:=\SL_{2}(\R)$
%be any discrete lattice.
%In this paper, we %are interested in
%study
%the pair-correlation of the angles $\gt_{\g}$, for $\g\in\G$.

\begin{figure}
\includegraphics[width=3in]{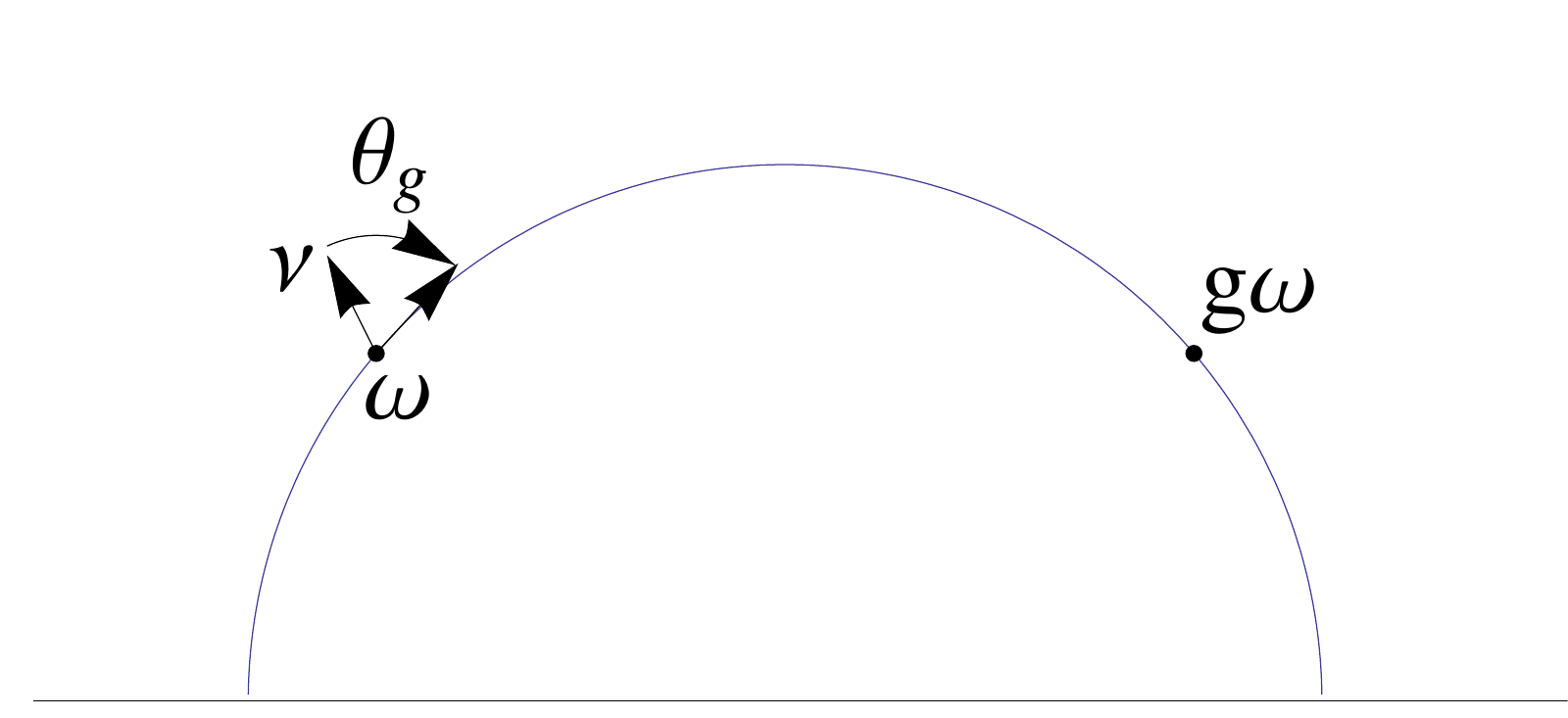}
\caption{%Given a point $\gw\in\bH$, unit direction $\nu$, and $g\in G$, let
The angle $\gt_{g}$ %be
is
 the directed angle between $\nu$ and the tangent vector at $\gw$ of the geodesic connecting $\gw$ to $g \gw$.}
\label{fig:1}
\end{figure}

%
%As usual, d

Define
\be\label{eq:gNorm}
\|g\|^{2}:=2\cosh d(\gw,g\gw),
\ee
% %%%% FIXED %%% \footnote{Hmm... this isn't consistent with what we want for $g$ small, relative to the previous section... When $g$ is near the identity, $\cosh0=1$, not $0$... I'll leave it for now, but we need to figure out consistent notation later...}
where $d(\cdot,\cdot)$ is the hyperbolic distance.
 Setting
\be\label{eq:BQdef}
B_{Q}:=
%\bigg|
\left\{
g\in G:
\|g\|<Q
\right\}
%\bigg|
,
\ee
and normalizing $dg$ so that
\be\label{eq:dgNorm}
\vol(B_{Q})\sim \pi Q^{2},\qquad\qquad
(Q\to\infty)
\ee
it is well-known  %(see e.g., \cite{Patterson1975B})
that
\be\label{eq:GcapBQ}
\#\G\cap B_{Q}\sim
{\vol(B_{Q})
\over
V_{\G}
}
\sim
{\pi Q^{2}\over V_{\G}} %C
,
\qquad\qquad
(Q\to\infty)
\ee
where
\be\label{eq:VGdef}
V_{\G}:=
\vol(%\G\bk
G/\G)
.
\ee
Hence for $\g\in\G\cap B_{Q}$, the average spacing of the angles $\gt_{\g}$
 %in a ball
 is
 $$%\sim
 {2\pi}\cdot \frac{ V_\G}{\pi Q^2},
 $$
% of order $Q^{-2}$,
% where
%To study the pair correlation, t
and we should consider,
%To this end,
 for fixed $\xi>0$ and $Q\to\infty$,  the correlation of pairs of angles % restricted to growing norm balls,
 by defining
\begin{comment}
\beann
\cN_{Q}(\xi)
&:=&
\left|
\left\{
(\g,\g')\in\G^{2}:
\g\neq\g',\
%0\le
 \|\g\|,\|\g'\|<Q,\
|\gt_{\g}-\gt_{\g'}|<\tfrac{2 V_\G}{Q^2}\xi
\right\}
\right|
.
\eeann
% This explains our condition that the angles differ by at most $\xi/Q^{2}$.
%
\end{comment}
$$
\cN_{Q}(\xi)
:=
\foh
\left|
\left\{
(\g,\g')\in\G^{2}:
\g\gw\neq\g'\gw,\
%0\le
 \|\g\|,\|\g'\|<Q,\
|\gt_{\g}-\gt_{\g'}|<\tfrac{2 V_\G}{Q^2}\xi
\right\}
\right|
.
$$
The constant $\foh$ in front is to account for the symmetry in $\g,\g'$.
%(%We could moreover assume % (by conjugation)
%that $(\gw,\nu)=(i,\uparrow)$, but i
%In fact, i
%It is clear that
Note that while $\nu$ is needed to define $\gt_{g}$,
the difference
$
|\gt_{\g}-\gt_{\g'}|
,
$
defined as the distance to $2\pi\Z$,
is independent of $\nu$.

The pair correlation distribution function is then defined as % the limit
\be\label{eq:R2def}
R_2(\xi):=\lim_{Q\to\infty}
\frac{V_\G}{\pi Q^2}\cdot\cN_Q(\xi)
,
\ee
if the limit exists. If $R_{2}$ is moreover differentiable, then its derivative defines the pair correlation density,
\be\label{eq:g2Is}
g_{2}(\xi):={d\over d\xi}R_{2}(\xi).
\ee

%\begin{comment}
The difficulty in this setting of determining the pair correlation %distribution and
density is illustrated %numerically
in Figure \ref{fig:2a}. %Here we have again taken $\gw=i$ and $\G$ as in \eqref{eq:GamUniform}, and plotted the empirical pair correlation density versus the theoretical one detailed below.
The main  goal of this paper  is to  explain this picture. %

\begin{figure}
\includegraphics[width=5in]{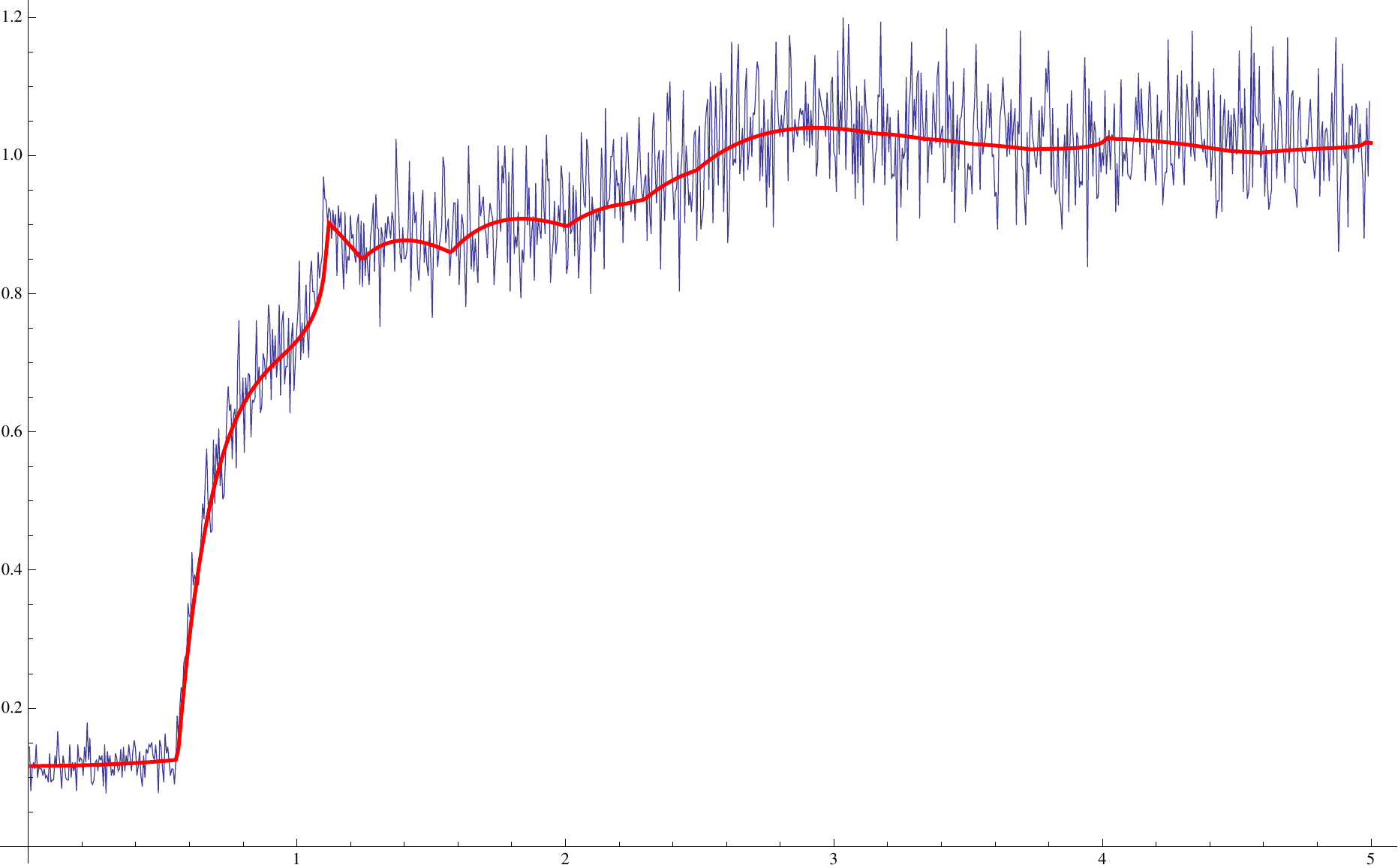}
\caption{The empirical pair correlation density $g_{2}(\xi)$ for $\gw$ and $\G$ as in Figure \ref{fig:1a} (with $Q=2\,000$), plotted against the function given on the right side of \eqref{eq:g2Def}. The fundamental domain in Figure \ref{fig:1a} is a hyperbolic octagon with all right angles, hence $V_{\G}=2\pi$ by Gauss-Bonnet. The cardinality of  $\G\cap B_{Q}$ is $2\, 000\, 914$, which beautifully matches \eqref{eq:GcapBQ}.}
\label{fig:2a}
\end{figure}

%\end{comment}
%
%
%\newpage
%
% In \cite{BocaPopaZaharescu2013}, Boca Popa and Zaharescu managed to prove the existence, and give an explicit formula, for the pair correlation density in the special case of $\G=\SL_2(Z)$ and an elliptic base point. Their work also led them to conjecture a formula for the pair correlation density for a general lattice and base point. The main result of this paper is a proof of their conjectured formula.
%
%
%\subsection{Statements of the Theorems}\
%
%$$R_2(\xi):=\lim_{Q\to\infty}\ \frac{V_{\G}}{
%%\#\G\cap B_{Q}
%Q^{2}
%}\cdot
%\cN_{Q}(\xi).
%$$
%and the pair correlation density function $g_{2}(\xi)$ is %given
%defined
%by
%$$
%g_{2}(\xi):={d \over d\xi}
%\left[
%\lim_{Q\to\infty}\frac1{Q^{2}}\cN_{Q}(\xi)
%\right].
%$$
%
%\
%
%result of this paper is the following

\begin{thm}\label{thm:1}
Let $\G< G$ be any lattice and $\gw\in\bH$ be any fixed base point.
The limit defining the pair correlation function $R_2(\xi)$ in \eqref{eq:R2def} exists, and is moreover differentiable. The density function
 %$$g_{2}(\xi):={d\over d\xi}R_2(\xi)$$
defined in \eqref{eq:g2Is} is given by the %explicit
formula
\be\label{eq:g2Def}
g_{2}\left(\frac{\xi}{V_\G}
\right)=\frac{V_\G}{2\pi}\sum_{M\in\G}
f_{\xi}(\ell(M))
.
\ee
Here
\be\label{eq:ellM}
\ell(M):=d(%i%
\gw
,M
%i
\gw
)=\arccosh\left({\|M\|^{2}\over2}\right),
\ee
and, on writing
\be\label{eq:ABCdef}
A=\cosh\ell,\quad
B=\sinh\ell,\quad
C=2\sinh(\ell/2)=\sqrt{2(A-1)}
,
\ee
the function $f_{\xi}$ % (\ell),$ $\ell>0$,
is given by
\be\label{eq:fXiDef}
%&&
%\hskip-.5in
%\\
%&&
f_{\xi}(\ell):=%\hskip-.25in
{2
\over  \xi^{2}}
\times
\threecase
{\ell,}
{if $B\le\xi$,}
{
\ell
+
\log(1+\xi^{2})
-
2
\log%\left({
%(\cosh\ell+\sinh\ell)
%(1+\xi^{2})
%\over
\left(A+\sqrt{B^{2} -\xi^{2}}\right)
%^{2}
%}\right)
,}
{if $B\ge\xi\ge C$,}
{
\ell-
\log\left({
%\cosh\ell+\sinh\ell
%\over
A+\sqrt{B^{2} -\xi^{2}}}\right)
,}
{if $C\ge\xi%\le C
$.}
\ee
%Let $V$ denote the volume of $\G\bk\bH$. Then t
%With assumptions as above, l

Furthermore, there exists some $\gd>0$, depending only on the spectral gap for $\G$, so that, for fixed $\xi>0$ and $Q\to\infty$, we have
\be\label{eq:cNQrate}
\cN_{Q}(\xi)={\pi Q^{2}\over V_{\G}}
R_2(\xi) + O_{\xi}(Q^{2-\gd}).
\ee
\end{thm}

\begin{figure}
        \begin{subfigure}[t]{0.46\textwidth}
                \centering
		\includegraphics[width=\textwidth]{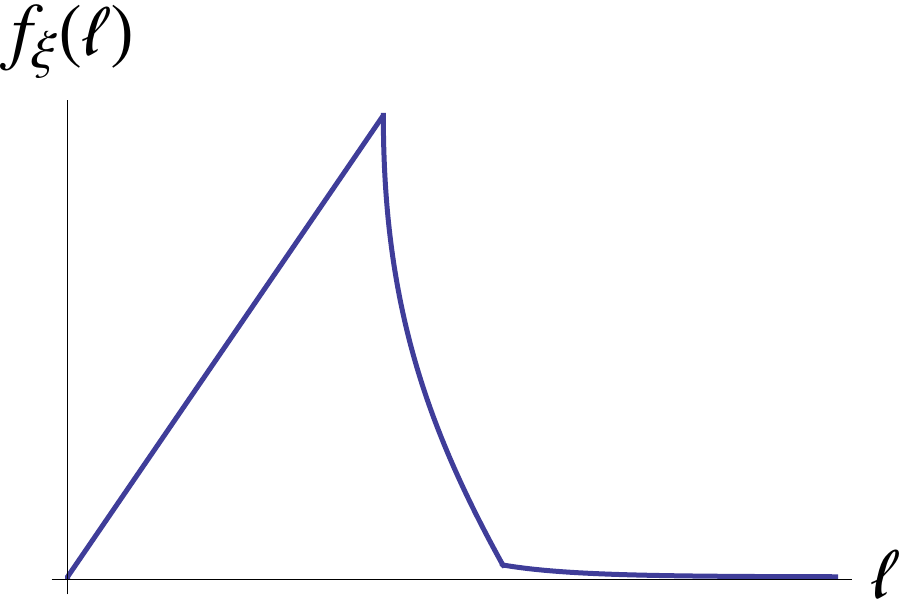}
                \caption{As a function of $\ell$ for $\xi$ fixed.}
                \label{fig:ell}
        \end{subfigure}%
%\qquad
\qquad
        \begin{subfigure}[t]{0.46\textwidth}
                \centering
		\includegraphics[width=\textwidth]{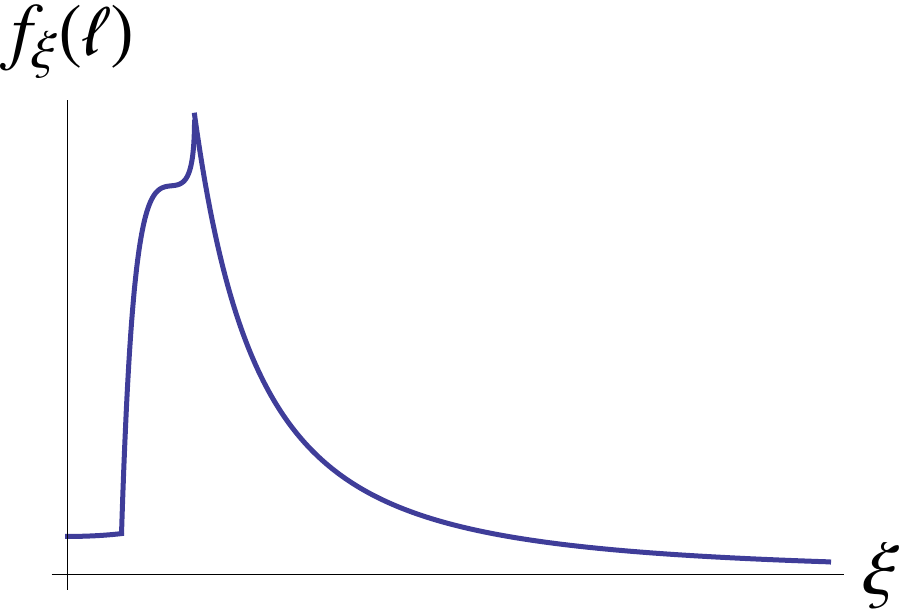}
                \caption{As a function of $\xi$ with $\ell$ fixed.}
                \label{fig:xi}
        \end{subfigure}
\caption{Typical plots of the function $f_{\xi}(\ell)$ in \eqref{eq:fXiDef}.}
\label{fig:2}
\end{figure}

\begin{rmk}
This confirms a conjecture
% (with slightly different normalizations)
 due to Boca, Popa, and Zaharescu \cite[Conjecture 1]{BocaPopaZaharescu2013}, who proved  Theorem \ref{thm:1} for the special case $\G=\SL_{2}(\Z)$ and $\gw$ an elliptic point. (They did not  give a rate, though their proof is in principle effective.)
%Miraculously, their proof is entirely elementary (in the sense of not using automorphic forms; it is based on repulsion arguments for Farey tesselations, and Weil's bound for Kloosterman sums), the downside
It is remarkable that their proof does not use any spectral theory (and is based instead on repulsion arguments for Farey tessellations and Weil's bound for Kloosterman sums), the downside being
that it does not apply to
general
 lattices $\G$, or base points $\gw$ other than $i$ or $e^{i\pi/3}$. It is not surprising, then, that our approach is completely different from theirs.
\end{rmk}

\begin{rmk}\label{rmk:fxiCont}
Note that
$f_{\xi}(\ell)$ is continuous, see Figure \ref{fig:2} for a plot. %It %also
%decays (f
%Once $C>\xi$, the function $f_{\xi}$ decays at the rate
For given $\xi_{0}>0$,  we %clearly
have
uniformly over $\xi\in(0,\xi_{0}]$
 the bound
\be\label{eq:fxiDecay}
f_{\xi}(\ell)\ll_{\xi_{0}}
\frac{1%\xi^{2}
}{e^{2\ell}}
,\qquad
\ell\to\infty,
\ee
%Then
whence the sum defining \eqref{eq:g2Def} easily converges, since $e^{2\ell(M)}\asymp \|M\|^{4}$.
\end{rmk}

\begin{rmk}
The sum on $M\in\G$ given in \eqref{eq:g2Def} can be reformulated via the trace formula in terms of the spectral expansion of $L^{2}(\G\bk\bH)$, but it does not seem to be expressible in a
more intrinsic way that does not depend so explicitly on either the %length
norm
 or Laplace spectrum of $\G$.
This explains the complicated function whose graph is illustrated in Figure \ref{fig:2a}: each $M\in\G$ contributes a peak from $f_{\xi}$ in Figure \ref{fig:xi}  to the sum in \eqref{eq:g2Def}.
%See Figure \ref{fig:3} ***?*** for a sample graph of the pair correlation density, compared against empirical data, for $\G$ being the spin cover of the integer special orthogonal group $\SO\_{F}(\Z)$ preserving the $\Q$-anisotropic indefinite rational quadratic form $F(x,y,z)=x^{2}+y^{2}-3z^{2}$.
%%Equivalently, $\G$ corresponds to the group of norm-one elements of the quaternion division algebra given by $\bu=a+bI+cJ+dK$ with $I^{2}=J^{2}=3$ and $K=\frac13IJ$.
\end{rmk}

\begin{rmk}\label{rmk:Jens}
While the limiting pair-correlation is highly non-universal, depending critically on both $\G$ and $\gw$, it does have %the
(at least)
one %important
trait of universality:
%As pointed out by a referee,
the method of proof of Theorem \ref{thm:1} easily extends to give the following more general statement. Let $\cI\subset\R/(2\pi\Z)$ be any fixed subinterval, and consider the pair-correlation restricted to $\cI$, that is, define
$$
\cN_{Q}^{(\cI)}(\xi)
:=
\foh
\left|
\left\{
(\g,\g')\in\G^{2}:\begin{array}{l}
\g\gw\neq\g'\gw,\
%0\le
 \gt_{\g},\gt_{\g'}\in\cI,\\
 \|\g\|,\|\g'\|<Q,\
|\gt_{\g}-\gt_{\g'}|<\tfrac{2 V_\G}{Q^2}\xi\end{array}
\right\}
\right|
.
$$
Then,  in analogy with \eqref{eq:R2def}, the limit as $Q\to\infty$ of
$
\dfrac{2\pi}{ |\cI|}
\cdot
\dfrac{V_{\G}}{ \pi Q^{2}}
\cdot
\cN_{Q}^{(\cI)}(\xi)
$
exists, and is also equal to the same function $R_{2}(\xi)$, independently of the choice of $\cI$.
%(The rate, of course, depends on $|\cI|$.)
%
This generalization, kindly suggested to us by both Jens Marklof and a referee, %(??????)
has the following interpretation: the pair correlation is the same regardless of which part of the ``sky'' we observe from our ``planet'' $\gw$.
\end{rmk}

%\newpage

\begin{rmk}
%Closed geodesics...
In \cite{BocaPasolPopaZaharescu2012}, another expression for $g_{2}$ is %established
given, again
for the case $\G=\SL_{2}(\Z)$ and $\gw=i$, in terms of lengths
of reciprocal geodesics on the modular surface. More generally, keeping the base point $\gw=i$, if we assume that $\G$ is invariant under transpose and there is another lattice $\G'$ such that the matrices $A=M^tM$ with $M\in \G$ are all the symmetric matrices in $\G'$, then any sum over any function $f(\ell(M))$ with $M\in \G$ can be written as a sum over $f(\ell(C)/2)$, where $C$ runs over closed geodesics in $\G'\bk \bH$ passing through $i$. Explicitly:
$$
\sum_{M\in \G}f(\ell(M))=|\G_\gw|\sum_C f(\ell(C)/2)
.
$$
Here $\G_{\gw}$ is the subgroup of $\G$ which stabilizes $\gw$.
Note that on the right we divide $\ell$ by 2, which compensates for the fact that the sum is over a much smaller set.
For a different base point $\omega$, the same is true but the requirement that $\G$ be transpose-invariant is replaced by a different involution corresponding to $\omega$.
\end{rmk}

\begin{rmk}
We make no attempt to optimize the rate in \eqref{eq:cNQrate}, as can surely be done with some effort. Our point is simply that the method is completely effective, with a power gain.
The value of $\gd$ coming from our proof is given as follows.
Let
$$
\gT
\in
%\left
(
0,\tfrac12
%\right
)
$$
be a spectral gap for $\G$, that is, a number so that the first non-zero eigenvalue $\gl_{1}$ of the hyperbolic Laplacian on $L^{2}(\G\bk\bH)$ satisfies
\be\label{eq:gl1}
\gl_{1}>\frac14-\gT^{2}.
\ee
(If $\G$ is arithmetic, then $\gT=7/64$ is known \cite{KimSarnak2003}.)
Then \eqref{eq:cNQrate} holds with any
\be\label{eq:cNQrateIs}
\gd<(1-2\gT)/26.
\ee
\end{rmk}

\begin{rmk}
It is interesting to compare our result to analogous results in the Euclidian setting. In this case one fixes a Euclidian lattice $\Lambda\subseteq \R^2$ and studies the distribution of angles between line segments connecting the origin (or a different point $\alpha\in \R^2/\Lambda$) to lattice points contained in increasing domains of $\R^2$. Here the angles become equidistributed on the circle (independently on the choice of $\alpha$) but the fine scale statistics depend on the choice of $\alpha$. In \cite{ElBazMarklofVinogradov2013}, for $\alpha$ satisfying certain diophantine properties the pair correlation was shown to be that of a Poisson process, in agreement with the average pair correlation previously computed in \cite{BocaZaharescu2006}. On the other hand, for $\alpha=0$, it is natural to consider primitive vectors in $\Z^2$, in which case the pair correlation density was explicitly computed in \cite{BocaZaharescu2005} and is far from  Poisson.
\end{rmk}

\begin{rmk}
%Even in
Returning to
hyperbolic space, one can formulate an alternate version of the problem,
%(these correspond in
similar to the Euclidean setting
 %analogue
 above % to taking
 with
 $\alpha\neq 0$.
Fixing {\it two}
 %another
 base points $\gw_{1}$ and  $\gw_{2}$, and $g\in G$,
% there is a natural generalization
%one can formulate an alternate version
%of
consider
%the geodesic angles given by the angle
the angle
$\theta_g(\omega_1,\omega_2)$ %, measuring
 between some fixed direction $\nu$ and the tangent vector at $\omega_1$ of the geodesic ray connecting $\omega_1$ to $g\omega_2$. The distribution of the angles $\theta_{\g}(\omega_1,\omega_2)$ with $\g\in \G$ was studied in \cite{Boca2007} when the angles are ordered by $d(\omega_1,\g\omega_1)$, and again in \cite{RisagerTruelsen2010} when ordered by $d(\omega_1,\g\omega_2)$. In the second ordering these angles become equidistributed with respect to Lebesgue measure (in fact, by conjugating the lattice this is reduced to the case of $\omega_1=\omega_2$). However, in the first ordering they become equidistributed with respect to a different measure, $\rho_{\omega_1,\omega_2}(\theta)d\theta$, depending on the base points. It would thus be interesting to study the pair correlation
also for the first
 %orderings for these %generalized
% notions of
ordering. (Note that when the angles themselves are not uniformly distributed, one must ``renormalize'' the pair correlation function to have mean spacing one everywhere.)
\end{rmk}

%%%%%%%%%%%%%%%%%%%%%%%%%%%%%
\begin{comment}
\begin{rmk}
The method of proof should allow the condition that $\G$ be a lattice to be relaxed to a geometrically finite discrete group whose critical exponent is not too small. Note that in this case the angles are equidistributed with respect to the Patterson-Sullivan measure on the circle which is singular with respect to Lebesgue measure (of course the answer will now also depend on the Patterson-Sullivan measure.) We do not pursue this generalization here.
\end{rmk}
\end{comment}
%%%%%%%%%%%%%%%%%%%%%%%%%%%%%%%

It is % of
interesting
to examine the boundary behavior of the pair correlation density function $g_{2}(\xi)$. For $\xi\to0$, it follows immediately from \eqref{eq:g2Def} and l'Hopital's rule that
$$
g_{2}(0)={V_{\G}\over \pi}\sum_{M\in\G\atop\ell(M)>0}{1\over e^{2\ell(M)}-1},
$$
as observed in \cite[(1.3)]{BocaPopaZaharescu2013}. In particular, $g_2(0)>0$ which is in contrast to the result in the Euclidian setting studied in \cite{BocaZaharescu2005}, where the pair correlation vanishes near zero.
For the other extreme, $\xi\to\infty$, in many natural settings,
the pair correlation density function approaches $1$% at infinity
;
see Figure \ref{fig:2a}.
We confirm the conjecture in \cite[(1.4)]{BocaPopaZaharescu2013} that this case is no different.
\begin{thm}\label{thm:2}
As $\xi\to\infty$,
$$
%{***\over \pi %\xi^{2}
%}
%\sum_{M\in\G}f_{\xi}(\ell(M))
g_{2}\left({\xi
%\over V_{\G}
}
\right)
=
1+O\left(\frac{1}{\xi^{(1-2\gT)/3}}\right),
$$
where $\gT\in(0,\foh)$ is a spectral gap for $\G$.
\end{thm}

\subsection{Outline}\

The method of proof, and the  rest of the paper, proceed as follows.
%%%%%%%%%%%%%%%%%%%%%%%%%%%%%%%
\begin{comment}
First,
 and using the fact that all conditions are symmetric in $\gamma$ and $\gamma'$, we can write
$$
\cN_{Q}(\tfrac{\xi}{V_\G})
=\frac{1}{2}
\left|
\left\{
(\g,\g')\in\G^{2}:
\g\neq\g',\
%0\le
 \|\g\|,\|\g'\|<Q,\
%0<
|\gt_{\g}-\gt_{\g'}|<\tfrac{2 \xi}{Q^2}
\right\}
\right|
.
$$
\end{comment}
%%%%%%%%%%%%%%%%%%%%%%%%%%%%%%%%%%%%%%%
Following Boca-Pasol-Popa-Zaharescu \cite{BocaPasolPopaZaharescu2012}, we first replace $\g'$ in $\cN_{Q}(\xi)$ by the
variable $M=\g^{-1}\g'$,
to measure  %the ``distance'' from
how far
$(\g,\g')$ is ``off-diagonal''.
%We observe that the condition on $M$ is %then
%invariant on the right by
Let the stabilizer of $\gw$ be denoted by
$$K_{\gw}=\Stab_{G}(\gw)\cong
\mathrm{PSO}(2)
,
$$
which is a
maximal compact subgroup of $G$.
\begin{comment}
, and let
   $\G_{%i%
\gw
}=\G\cap
%\Stab_{G}(\gw)
K_{\gw}
$.
\end{comment}

Then  switching to the more convenient variable, $\tfrac{\xi}{V_\G}$, we may write
\beann
\cN_{Q}(\tfrac{\xi}{V_\G})
&=&
\frac{
%|\G_{%i%
%\gw
%}|
1}{2}
\sum_{{M\in
%%%%%%%%%%\G_{\gw}\bk
\G%/\G_{%i%
%\gw
}
\atop M%\neq I
\not\in
%\G_{%i%
%\gw
%}
K
}
\left|
\left\{
\g\in\G:
%\g\neq M^{-1},\,
%0\le
\|\g\|,\|\g M\|<Q,\,
%0<
|\gt_{\g}-\gt_{\g M}|<{2\xi\over Q^2}
\right\}
\right|
.
\eeann
%%%%%%%%%%%%%%%%%%%%%%%%%%%%%%%%%%%%%
\begin{comment}
For ease of exposition, we assume $\G_{%i%
\gw
}=\{I\}$; minor modifications are needed in the general case.
%
\end{comment}
%%%%%%%%%%%%%%%%%%%%%%%%%%%%%%%%%%%
After conjugating $\G$, we may %also
assume
 %further
  that $(\gw,\nu)=(i,\uparrow)$%. %
  , dropping all subscript $\gw$'s.

For each $M$, let
\be\label{eq:RMdef}
\cR_{M}(Q,\xi):=
\left\{
g\in G
:
\|g\|,\|g M\|<Q,\,
%0<
|\gt_{g}-\gt_{g M}|<{2\xi\over Q^2}
\right\}
,
\ee
be the region of interest,
so that
$$
\cN_{Q}(\tfrac{\xi}{V_\G})
=\frac{1}{2}\sum_{M\in\G\atop M\notin K}\#\G\cap\cR_{M}(Q,\xi)
.
$$

One can hope that
$$
\#\G\cap\cR_{M}(Q,\xi)\sim \frac{\vol(\cR_{M}(Q,\xi))}{
%\vol(\G\bk G)
V_{\G}
}
$$ %, and that this
can be proved using automorphic tools (spectral theory and dynamics). This is indeed the case for $\|M\|$ small,
 but for larger $M$, this volume can be of such small size that spectral methods are hopeless; the error term dominates the volume (see Proposition \ref{prop:cM} below).
Instead, once we are far enough off-diagonal, the entire contribution should be
%is
treated as
% error.
a remainder.

To this end, we introduce another parameter $T=T(Q)\to\infty$, and break $\cN_{Q}$ into  ``main'' and ``error'' terms according to whether or not $\|M\|<T$, writing
$$
\cN_{Q}(\tfrac{\xi}{V_\G})=\frac{1}{2}\sum_{\|M\|
%\leq
<
T\atop M\notin K}\#\G\cap\cR_{M}(Q,\xi)+\cE_{Q,T}(\xi),
$$
say.

After a few preliminary computations in \S\ref{sec:prelim},
we turn our attention
in
 \S\ref{sec:vol}
 to individual $M$'s with $\|M\|<T$, %and
 analyzing the
volumes of $\cR_{M}(Q,\xi)$. These are  the two most technically challenging sections, and we use \S\ref{sec:vol} to prove the following
\begin{prop}\label{prop:vol}
As $Q\to \infty$, we have that
$$
\vol(\cR_{M}(Q,\xi))={Q^{2}}
\int_{0}^{\xi}f_{\gz}(\ell(M))d\gz
+ O_{\xi}(\|M\|^{2}Q^{2/3})
,
$$
where $\ell(M)$ is defined in \eqref{eq:ellM}, and $f_{\gz}(\ell)$ is given in \eqref{eq:fXiDef}.
\end{prop}
The proposition is proved by a direct and delicate analysis of the region in $G$ corresponding to $\cR_{M}$. The estimate is evidently non-trivial only when $\|M\|=o(Q^{2/3})$.

In section \S\ref{sec:cM}, equipped with our understanding of these volumes, we prove %we
%using more-or-less standard automorphic methods
 %to prove
 the following
\begin{prop}\label{prop:cM}
Recalling from \eqref{eq:gl1} that $\gT\in(0,\foh)$ is a spectral gap for $\G$, we have
$$
\#\G\cap\cR_{M}(Q,\xi)= \frac{\vol(\cR_{M}(Q,\xi))}{V_{\G}}
+
O_{\xi}(
Q^{(17+2\gT)/9}
\|M\|^{16/9}
%+
%\|M\|^{2}
%Q^{2/3}
)
,
$$
as $Q\to\infty$.
\end{prop}

We are %a bit
quite
crude here in the estimation of the error, but it more than suffices for our applications, so we do not pursue the issue. The idea of the proof is a more-or-less standard smoothing and un-smoothing argument, though the execution of the method requires a bit of care.

Next we spend \S\ref{sec:cE}
disposing of the error, by proving the following
\begin{prop}\label{prop:cE}
For fixed $\xi>0$ and $T<Q,$
$$
\cE_{Q,T}(\xi)
%:=
%\sum_{M\in\G\setminus\{I\}\atop \|M\|<T}\#\G\cap\cR_{M}(Q,\xi)
\ll_{\xi}
Q^{2} \cdot\frac {\log Q}{T^{2}}
,
$$
as $T,Q\to\infty$.
\end{prop}
This is proved by returning to $\g'$ and the double sum, estimating directly the number of $(\g,\g')\in\G^{2}$ with
$$
 \|\g\|,\|\g'\|<Q,\quad
 \|\g^{-1}\g'\|\ge T,\quad\text{ and }
%0<
|\gt_{\g}-\gt_{\g'}|<{2\xi\over Q^2}
.
$$

Combining these ingredients, we prove Theorem \ref{thm:1}  in \S\ref{sec:thm1}.
Theorem  \ref{thm:2} is proved in \S\ref{sec:thm2} by observing that $g_2(\tfrac{\xi}{V_\G})$ is a multiple of an automorphic kernel
$$\sum_{M\in \G} f_\xi(d(\omega,M\omega)),$$
and then using a standard argument to show that, in the limit $\xi\to\infty$, this kernel is asymptotic to
$$\frac{1}{V_\G}\int_G  f_\xi(d(\omega,g\omega))dg.$$

\subsection{Notation}\

We use the following standard notation. The symbol $f\sim g$ means $f/g\to1$, and the notations $f\ll g$ and $f=O(g)$ are synonymous; moreover $f\asymp g$ means $f\ll g\ll f$.
Unless otherwise specified, the implied constants may depend at most on $\G$% (and hence $\gd$)
, which is treated as fixed.
The letter $c$ is a %fixed
positive constant, not necessarily the same at each occurrence.
 The symbol $\bo_{\{\cdot\}}$ is the indicator function of the event $\{\cdot\}$.
 The cardinality of a finite set $S$ is denoted $|S|$ or $\# S$.

\subsection*{Acknowledgements}\

We thank Alex Popa for help with the algorithm used to produce Figure \ref{fig:2a}. Thanks also to Jens Marklof and the referees for many detailed comments and suggestions on an earlier version.

\newpage

\section{Preliminary Computations}\label{sec:prelim}

In this section, we record a number of computations which will be useful in the sequel.
Recall that $\G$ is an arbitrary lattice in $G=\PSL_2(\R)$, %\SL_{2}(\R)$,
and we may assume $(\gw,\nu)=(i,\uparrow)$.

%As usual, let
Define the (semi)groups
\beann
K&:=&\mathrm{PSO}(2)=%\SO(2)=
\left\{
k_{\gt}:=\mattwo{\cos\tfrac{\gt}{2}}{\sin\tfrac{\gt}{2}}{-\sin\tfrac{\gt}{2}}{\cos\tfrac{\gt}{2}}
:
%-\pi\le \gt<\pi
\gt\in\R/(2\pi\Z)
\right\}
,
\\
A^{+}&:=&
\left\{
a_{t}:=\mattwo{e^{t/2}}{}{}{e^{-t/2}}
:
t\ge0
\right\}
.
\eeann
Taking a Cartan decomposition of $G=KA^{+}K$
% (here $K=K_{\gw}$)
 and writing (uniquely for $g\not\in K$)
\be\label{eq:KAKcoords}
g=k_{\gt(g)%^{(1)}
%_{g}
}a_{t(g)
%_{g}
}k_{
%\gt^{(2)}_{g}
\vf(g)
},
\ee
we see that the angle in \eqref{eq:angle} is given simply by
\be\label{eq:gtGgt1}
\gt_{g}(\gw,\nu)=\gt
%^{(1)}_{
(g)%}
.
\ee

Next, for $\xi>0$, $Q\to\infty$ and $M\in\G$, we study the region $R_{M}(Q,\xi)$ given in \eqref{eq:RMdef}.
It will be convenient to parametrize these conditions
with explicit coordinates.

\begin{lem}
Fix $M\in G$, $M\not\in K$, and write
$M$ in the  Cartan decomposition
\be\label{eq:Mcoords}
M
=
k_{m}a_{\ell}k_{*}
,
\ee
so that $\ell=\ell(M)$ in the notation of \eqref{eq:ellM}.
For $g\in G$, $g\not\in K$, take the decomposition $G=KA^{+}Kk_{-m}$, writing
\be\label{eq:gCoords}
g
=
k_{\gt_{g}}a_{t}k_{\vf}k_{-m}
.
\ee
Then recalling the notation \eqref{eq:ABCdef}, we have
\be\label{eq:BQ}
%\boxed{
%2\cosh t<Q^{2}
\|g\|^{2}=2\cosh t
%}
,
\ee
\be\label{eq:R1}
\|gM\|^{2}=
2 (
A%\cosh\ell
 \cosh t +B \cos(\vf)
% \sinh\ell
 \sinh t)
 ,
\ee
and
\be\label{eq:R2}
\tan(\gt_{gM}-\gt_{g})
=
\frac{B %\sinh\ell
\sin (\vf)
   }{
A\sinh t
%\cosh\ell
+ B\cos(\vf) \cosh t
%\sinh \ell
   }
   .
\ee
\end{lem}

\pf
The equation \eqref{eq:BQ} is immediate from the definition \eqref{eq:gNorm}.
For \eqref{eq:R1},
we observe that
\beann
gM
&=&
(k_{\gt_{g}}
a_{t}k_{\vf}k_{-m})
(k_{m}a_{\ell}
k_{*})
=
k_{\gt_{g}}
(a_{t}k_{\vf}
%k_{-m}
%k_{m}
a_{\ell}
)k_{*}
,
\eeann
and take norms:
\beann
\|gM\|^{2}
&=&
\|
a_{t}k_{\vf}a_{\ell}\|^{2}
\\
&=&
2 (\cosh(\ell + t)\cos^2\tfrac{\vf}{2}  + \cosh(\ell - t)\sin^2\tfrac{\vf}{2} )
\\
&=&
2 (\cosh\ell \cosh t + \cos \vf \sinh\ell \sinh t)
,
\eeann
as claimed.

To recover $\gt_{gM}=\gt_{gM}(\gw,\nu)$, or rather $\gt_{gM}-\gt_{g}$, we need to compute the left ``$K$'' in the $KA^{+}K$ decomposition of $(k_{\gt_{g}})^{-1}gM$. To accomplish this, we
act on $i$,
 letting
 $$
 z:=(k_{\gt_{g}})^{-1}gM
 %\circ
 \cdot
 i
 =
 a_{t}k_{\vf}
a_{\ell}
%\circ
\cdot
i
,
 $$
 and send $\bH\to\bD$ via $z\mapsto w={z-i\over z+i}$. Then the left hand side of \eqref{eq:R2} is equal to $\Im(w)/\Re(w)$, and a computation shows that
\beann
\tan(\gt_{gM}-\gt_{g})
&=&
\frac{\left(e^{2 \ell}-1\right) e^t \sin (
   \vf)}{\left(e^{2 (\ell+t)}-1\right) \cos
   ^2(\tfrac{\vf}{2})-\left(e^{2 \ell}-e^{2 t}\right) \sin
   ^2(\tfrac{\vf}{2})}
  \\
&=&
%\boxed{
\frac{
\sinh\ell\sin (\vf)
   }{
\sinh t \cosh\ell + \cos(\vf) \cosh t \sinh \ell
   }
%   }
,
\eeann
whence the claim follows.
\epf

Since $\tan$ is not injective on $\R/(2\pi\Z)$, we need to be %a little
careful when extracting from a small value of \eqref{eq:R2} the implication that $\gt_{gM}-\gt_{g}$ is near $[0]$ rather than $[\pi]$. We accomplish this by an elementary argument in hyperbolic geometry (thanks to one of the referees for suggesting this approach).

\begin{lem}\label{lem:hyp}
Assume that $\|h\|\le\|g\|$. Then $|\gt_{g}-\gt_{gh}|\le \pi/2$.
\end{lem}
\pf
Assume on the contrary  that $|\gt_{g}-\gt_{gh}|>\pi/2$; then the hyperbolic triangle with vertices $i$, $g\cdot i$, $gh\cdot i$ has an obtuse angle at $i$, whence the side opposite $i$ is the longest.
In particular, $d(g\cdot i, gh\cdot i)>d(i,g\cdot i)$, which implies
$$
\|h\|^{2}=2\cosh d(i,h\cdot i)=
2\cosh d(g\cdot i, gh\cdot i)
>
2\cosh d( i, g\cdot i)
=
\|g\|^{2}
,
$$
contradicting the assumption.
\epf

%\newpage

Next we record an estimate of how much $\theta(g),\vf(g)$, and $\|g\|$ change when $g$ is multiplied by an element from
$KA_{\gd}K$,
where
$$
A_\delta=\{a_\ell: |\ell|<\delta\}.
$$
Recalling the notation \eqref{eq:gNorm}--\eqref{eq:BQdef}, observe that this set is the same as
\be\label{eq:Ddel}
B_{\delta_{1}}=KA_\delta K
,
\ee
where
\be\label{eq:gdTogd1}
\gd_{1}^{2}
=
2\cosh\gd
.
\ee
\begin{lem} \label{lem:pert}
Let $g\in G$ with $\|g\|>3$ and $h\in % KA_\delta K
B_{\gd_{1}}
$. For small $\delta> 0$, when multiplying from the right we have
\bea\label{eq:dnormR}
\|gh\|&=&\|g\|(1+O(\delta)),
\qquad t(gh)\ =\ t(g)+O(\gd),
\eea
\bea
\label{eq:dgtR}
|\theta_g-\theta_{g h}|&=&O(\tfrac{\delta}{\|g\|^2}),
\eea
and when multiplying from the left,
\bea\label{eq:dnormL}
\|hg\|&=&\|g\|(1+O(\delta)),
\qquad t(hg)\ =\ t(g)+O(\gd),
\eea
\bea
\label{eq:dvfL}
|\vf(g)-\vf(hg)|&=&O(\tfrac{\delta}{\|g\|^2}).
\eea
\end{lem}
\begin{proof}
First note that  $\gt(g^{-1})=\pi-\vf(g)$,
since to remain in $A^{+}$ we must write
$$
(k_{\gt}a_{t}k_{\vf})^{-1}
=
k_{-\vf}a_{-t}k_{-\gt}
=
k_{-\vf}k_{\pi}a_{t}k_{-\pi}k_{-\gt}.
$$
Then using this and
$\|g\|=\|g^{-1}\|$,
it is enough to prove the results when $h$ acts on the right.

Now, write $g=k_{\gt}a_tk_{*}$ with $t>1$ and
$h=k_{*}a_s k_{*}$ with $|s|<\delta$ so that, say,
$gh=k_{\gt}a_{t}k_{\vf}a_{s}k_{*}$.
Using \eqref{eq:R1} we have
\begin{eqnarray*}
\|gh\|^2&=&2\cosh(t)\cosh(s)+2\cos(\vf)\sinh(t)\sinh(s)\\
&=&\|g\|^2(1+O(\delta^2))+O(\delta\|g\|^2)\\
&=& \|g\|^2(1+O(\delta)),
\end{eqnarray*}
which, after taking square roots, implies the first equality in \eqref{eq:dnormR}. For the second, we have from the above that
$$
2\cosh t(gh)=2\cosh t(g)+O(\gd\|g\|^{2}),
$$
and hence
$$
t(gh)
=\arccosh(\cosh t(g)+O(\gd\|g\|^{2}))
= t(g)+O\left({\gd\|g\|^{2}\over\sqrt{\cosh^{2}t(g)-1}}\right).
$$
This implies \eqref{eq:dnormR} since $t(g)>1$.

%\newpage

%Next,
Finally,
%using \eqref{eq:R2} we get
we see from \eqref{eq:R2}
that
\begin{eqnarray*}|\tan(\gt_{g}-\gt_{gh})|&=&
\frac{
|\sinh(s)\sin (\vf)|}{|\sinh t \cosh s + \cos(\vf) \cosh t \sinh s |}\\
&<&\frac{
\delta}{2\sinh t  -\delta \cosh(t)}\\
&=& \frac{\delta}{\|g\|^2}\frac{1}{\tanh(t)-\delta/2}
%=O(
\ll
\frac{\delta}{\|g\|^2}%)
.
\end{eqnarray*}
This implies \eqref{eq:dgtR} by Lemma \ref{lem:hyp}.
\end{proof}

Note that if we multiply $g$ from the right
by an element of $B_{\delta_1}$ %=KA_\delta K$
we
clearly
 have no control over how much $\vf(g)$ changes.
Instead, for  $\gd>0$ small, we define a small $\gd$-ball
\be\label{eq:Ddel1}
D_{\gd}:=K_{\gd}A_{\gd}K_{\gd},
\ee
where
$$
K_{\gd}:=\{k_{\gt}\in K:|\gt|<\gd\}.
$$
The next lemma gives the desired control.
% shows that, when multiplying on the right by an element of $D_\delta$, % on the right or left,
%Now, when multiplying by an element of $D_\delta$ on the right,
 %(on either side)
 %all
%this parameter
%$\vf$
 %$\theta,\vf$, and $t$
% changes by $O(\delta)$.

\begin{lem}
Let $g\in G$ with $\|g\|>3$ and $h\in D_{\gd}$. For small $\delta> 0$, when multiplying from the right we have
\be\label{eq:vfgh}
\vf(gh)=\vf(g)+O(\gd).
%,\qquad
%t(gh)=t(g)+O(\gd)=t(hg)+O(\gd).
\ee
\end{lem}
\pf
It is clearly enough to consider $h=a_{s}\in A_{\gd}$.
We will show that  $\gt(hg)=\gt(g)+O(\gd)$, from which \eqref{eq:vfgh} follows on taking inverses. Writing $g=k_{\gt}a_{t}k_{*}$, we again study the angle in the disk model $\bD$, setting $z=hg\cdot i$ and $w=(z-i)/(z+i)$.
%The same
As before, a
calculation
%as before
shows that
$$
\cos(\gt(hg))
=
\frac{  \cos \theta   +\tanh s
   \coth t}{\sqrt{  (\cos \theta
   +\tanh s \coth t)^2+
   \sin ^2\theta }}
=
%\frac{
 \cos \theta   +O(\gd)
 %}{\sqrt{  (\cos \theta
 %  +O(\gd))^2+
 %  \sin ^2\theta }}
 ,
$$
where we used that $t>1$. A similar identity holds for $\sin(\gt(hg))$, whence we are done.
\epf

Next we estimate how equations \eqref{eq:BQ}--\eqref{eq:R2} are affected under simultaneous left-$B_{\gd_{1}}$ and right-$D_{\gd}$ perturbations. % on both sides.

%Since we are mainly concerned with $\gt$, we can allow ourselves to multiply just by $B_{\gd_{1}}$ on the left.
%Then for $g\in G$,  let
%$$
%g_{1}\in B_{\gd_1}\, g\, D_{\gd}.
%$$
%For $M\in \G,$ $M\not\in K$, write $M$ as in \eqref{eq:Mcoords}, write $g$ as in \eqref{eq:gCoords} and write
%$$
%g_{1}=k_{\gt_{g_{1}}}a_{t_{1}}k_{\vf_{1}}k_{-m}.
%$$

\begin{lem}\label{lem:pert2}
Fix $g,M\in G$, $M\not\in K$, with %$\|g\|>3$ and
$\|g\|\geq 10\|M\|%>10
$. % *** Need $\|M\|>10$ also?***
Then, for any
$
g_{1}\in B_{\gd_1}\, g\, D_{\gd},
$
with sufficiently small $\delta> 0$, we have
%\bea
%\label{eq:t1s}
%%\gt_{g_{1}}&=&\gt_{g}+O(\gd)\\
%t_{g_1}&=&t(g)+O(\gd)\\
%\label{eq:t3s}
%\vf(g_1)&=&\vf(g)+O(\gd),
%\eea
%\be\label{eq:gCh}
%\|g_{1}\|=\|g\|(1+O(\gd)),
%\ee
\be\label{eq:gMch}
\|g_{1}M\|^2
=
\|gM\|^2
+
 O\left(
\gd \|g\|^2\|M\|^2
\right)
,
\ee
and
\be\label{eq:tanCh}
\tan(\gt_{g_{1}}-\gt_{g_{1}M})
=
\tan(\gt_{g}-\gt_{gM})
+
O\left(
{\gd \|M\|^4\over \|g\|^{2}}
\right)
.
\ee
\end{lem}
\pf
%First, noting that $g_1\in B_{\delta_1}gD_\delta\subseteq B_{\delta_1}g B_{\delta_1}$ we see that \eqref{eq:gCh} follows directly from Lemma \ref{lem:pert}.

First we note that, on
writing $g_1=h_1gh_2$ with $h_1\in B_{\delta_1}$ and $h_2\in D_{\delta}$, we %clearly
 have
 \be\label{eq:phig1TophiG}
\vf(g_{1})=
 \vf(h_1gh_2)
 =\vf(g h_{2})+O(\delta/\|g h_{2}\|^{2})
 =\vf(g )+O(\delta)
,
 \ee
where we used
\eqref{eq:dvfL}, \eqref{eq:dnormR}, \eqref{eq:vfgh}, and $\|g\|\gg1$.

Next to deal with the norm of $gM$, recall again the notation \eqref{eq:ABCdef} and, in light of \eqref{eq:R1}, consider the function
$$
F_1(\vf,t)=2(A\cosh t+B\cos(\vf)\sinh t)
.
$$
Then  writing $gM=k_\theta a_t k_{\vf}a_\ell k_*$ and $g_1M=k_{\theta_1} a_{t_1} k_{\vf_1}a_\ell k_*$, we have
$$
\|gM\|^2=F_1(\vf,t),\qquad {\rm{and}}\qquad \|g_1M\|^2=F_1(\vf_1,t_1).
$$
The partial derivatives $|\frac{\partial F_1}{\partial \vf}(\vf,t)|$ and $|\frac{\partial F_1}{\partial t}(\vf,t)|$ are easily seen to be bounded by $O(A\cosh(t))$. So using
\eqref{eq:dnormR} and \eqref{eq:dnormL} that $t_1=t+O(\delta)$,
%and
together with
 \eqref{eq:phig1TophiG}, we see that
%\begin{eqnarray*}
%|\frac{\partial F_1}{\partial \vf}(\vf,t)| &=& |2\sin(\vf/2)\cos(\vf/2)(\cosh(\ell+t)+\cosh(\ell-t)|\\
%\leq  4\cosh(\ell)\cosh(t)
%\end{eqnarray*}
%\begin{eqnarray*}
%|\frac{\partial F_1}{\partial t}(\vf,t)| |2\cos^2(\vf/2)\sinh(\ell+t)-2\sinh^2(\vf/2)\sinh(\ell-t)|\\
%\leq 4\cosh(\ell)\cosh(t)
%\end{eqnarray*}
%imply that
$$
|\|g_1M\|^2-\|gM\|^2|=|F_1(\vf_1,t_1)-F_1(\vf,t)|\ll \delta \|M\|^2\|g\|^2,
$$
%Note that, using that $\|g\|\ll \|gM\|\|M\|$ we also get that
%$|\|g_1M\|^2-\|gM\|^2|=O(\delta \|gM\|^2\|M\|^4)$.
giving \eqref{eq:gMch}.

Similarly, to deal with $\tan(\theta_{gM}-\theta_{g})$ we consider from \eqref{eq:R2} the function
$$
F_2(\vf,t)
=
\frac{B\sin(\vf)}{A\sinh t+B\cos(\vf)\cosh t}
,
$$
so that
$$\tan(\theta_{gM}-\theta_g)=F_2(\vf,t).$$
We will show that both partial derivatives $\frac{\partial F_2}{\partial \vf}(\vf,t),\frac{\partial F_2}{\partial t}(\vf,t)$ are bounded by $O(\frac{e^{2\ell}}{\cosh(t)})$ implying that
\begin{eqnarray*}
%|\tan(\theta_{gM}-\theta_g)-\tan(\theta_{g_1M}-\theta_{g_1})|&=&\\
F_2(\vf_1,t_1)=F_2(\vf,t)+O(\delta \tfrac{\|M\|^4}{\|g\|^2}).
\end{eqnarray*}

To bound the partial derivatives, a simple calculation gives
$$
\left|\frac{\partial F_2}{\partial \vf}(\vf,t)\right|
=
\frac{1}{\cosh(t)}\cdot
{\tanh^{2}\ell\over \tanh^{2}t}\cdot
\frac{|1+\frac{\tanh t}{\tanh \ell}\cos(\vf)|}{|1+\frac{\tanh\ell}{\tanh t}\cos(\vf)|^2}
.
$$
We have assumed that $\|g\|>\|M\|>1$; hence we have $t>\ell>0$, and $0<{\tanh \ell\over\tanh t}<1$.
For $X\in(0,1)$, the function
$$
{|1+\frac1X \cos(\vf)|\over |1+X\cos(\vf)|^{2}}
$$
is maximized at $\cos(\vf)=-1$, with maximum value
$
{1\over X(1-X)}
.
$
Hence
$$
\left|\frac{\partial F_2}{\partial \vf}(\vf,t)\right|
\le
\frac{1}{\cosh(t)}\cdot
{\tanh\ell\over \tanh t-\tanh\ell}
.
$$
Assuming further that $\|g\|\ge10\|M\|$ gives $t>\ell+1$, and
\be\label{eq:tanhtToell}
\tanh t-\tanh\ell>\tanh(\ell+1)-\tanh\ell>\foh e^{-2\ell},
\ee
whence
$$
\left|
\frac{\partial F_2}{\partial \vf}(\vf,t)
\right|
<  \frac{2e^{2\ell}}{\cosh(t)}
,
$$
as desired.

For the second partial derivative, we have
\beann
\left|
\frac{\partial F_2}{\partial t}(\vf,t)
\right|
&=&
{|\sin(\vf)|\over \sinh t}
\cdot
{\tanh\ell\over\tanh t}
\cdot
{
|1+
\tanh^{2} \ell\
{\tanh t\over\tanh \ell}
\cos(\vf)|
\over
|
1+{\tanh \ell\over \tanh t}\cos(\vf)
|^{2}
}
.
\eeann
It is easy to see that if $0<Y<X<1$ with
\be\label{eq:YtoX}
1<{2X^{2}\over Y(1+X)},
\ee
then the function
$$
{
|1+{Y\over X}\cos(\vf)|
\over
|
1+X\cos(\vf)
|^{2}
}
$$
attains its maximum value of
$$
{
1-{Y\over X}
\over
(
1-X
)^{2}
}
$$
at $\cos(\vf)=-1$.
Letting $Y=\tanh^{2}\ell$ and $X={\tanh\ell\over\tanh t}$, it is clear that $0<Y<X<1$. We also have
$$
1<{2\over \tanh t\ (\tanh t+1)}
<{2\over \tanh t\ (\tanh t+\tanh \ell)}
=
{2({\tanh\ell\over \tanh t})^{2}\over \tanh^{2}\ell\ (1+{\tanh\ell\over\tanh t})}
,
$$
whence \eqref{eq:YtoX} is satisfied.
We can thus estimate
\beann
\left|
\frac{\partial F_2}{\partial t}(\vf,t)
\right|
&\le&
{\tanh\ell
\over \cosh t}
\cdot
{
1-
\tanh \ell
\tanh t
\over
(
\tanh t-\tanh \ell
)^{2}
}
.
\eeann
Now we use \eqref{eq:tanhtToell} and
$$
1-\tanh\ell\tanh t<1-\tanh\ell\tanh(\ell+1)<3e^{-2\ell}
$$
to finish the proof of \eqref{eq:tanCh}.
\epf

We conclude this section with a few computations regarding the function $f_\xi$ in \eqref{eq:fXiDef}; these are needed
 %that we will need
 for %proving
the proof of
Theorem \ref{thm:2}.
\begin{lem}\label{lem:fxiInt}
As $\xi\to\infty$ we have
\be\label{eq:fxiInt}
\int_G f_\xi(\ell(g))dg=2\pi+O\left(\frac{1}{\xi^2}\right),
\ee
and for any fixed $\alpha\in (0,1)$,
\be\label{eq:fxiBound}
\int_G f_\xi(\ell(g))\|g\|^{-\alpha}dg\ll_{\ga} \frac{1}{\xi^{\alpha}}.
\ee
\end{lem}
\begin{proof}
In the $KA^+K$ coordinates \eqref{eq:KAKcoords}, we have $\ell(g)=t(g)$ and a computation shows that the assumption \eqref{eq:dgNorm} forces the normalization
\be\label{eq:dgIs}
dg=\frac1{2\pi}d\gt\,\sinh t\, dt\ d\vf.
\ee
We thus have that
\beann
\int_G f_\xi(\ell(g))dg &=& 2\pi\int_0^\infty f_\xi(t)\sinh(t)dt\\
&=& -2\pi\int_0^\infty f_\xi'(t)\cosh(t)dt
\eeann

The derivative, $f_\xi'(\ell)$, is given by
\be\label{eq:fxi'}
f'_{\xi}(\ell)=%\hskip-.25in
{2
\over  \xi^{2}}
\times
\threecase
{1,}
{if $\ell< \ell_1(\xi)$,}
{
1
-
{2\sinh(\ell) \over \sqrt{\sinh^2(\ell)-\xi^2}},}
{if $\ell_1(\xi)< \ell< \ell_2(\xi) $,}
{1-
{\sinh(\ell) \over \sqrt{\sinh^2(\ell)-\xi^2}},}
{if $\ell> \ell_2(\xi)%\le C
$,}
\ee
where the two points of discontinuity $\ell_1(\xi)$ and $\ell_2(\xi)$ satisfy
\begin{equation}\label{eq:fxidisc}
\sinh(\ell_1(\xi))=2\sinh(\tfrac{\ell_2(\xi)}{2})=\xi.
\end{equation}
Plugging this into the integral, and bounding $f_\xi'(t)=O(\frac{1}{\sinh^2(t)})$ for $t>\ell_2(\xi)$, we obtain
\beann
\int_G f_\xi(\ell(g))dg&=& -\frac{4\pi}{\xi^2}\int_0^{\ell_2(\xi)}
%\sinh(t)
\cosh(t)
dt\\
&+& \frac{4\pi}{\xi^2}\int_{\ell_1(\xi)}^{\ell_2(\xi)} \frac{2\sinh(t)\cosh(t)}{\sqrt{\sinh^2(t)-\xi^2}}dt+O\left(\frac{1}{\sinh(\ell_2(\xi))}\right)\\
&=& \frac{4\pi}{ \xi^2}\left(2\sqrt{\sinh^2(\ell_2(\xi))-\xi^2}-\sinh(\ell_2(\xi))\right)+O\left(\frac{1}{\sinh(\ell_2(\xi))}\right)\\
&=& \frac{4\pi\sinh(\ell_2(\xi))}{\xi^2}+O\left(\frac{1}{\sinh(\ell_2(\xi))}\right)=2\pi+O\left(\frac{1}{\xi^2}\right).
\eeann

For the second statement, similarly
\beann
\int_G f_\xi(\ell(g))\|g\|^{-\alpha}dg &=& 2\pi\int_0^\infty f_\xi(t)\frac{\sinh(t)}{\cosh^{\alpha/2}(t)}dt\\
&\ll_{\ga} & \int_0^{\ell_2(\xi)}f_\xi'(t)\cosh^{1-\alpha/2}(t)dt+%O(
\frac{1}{\xi^2}%)
\\
&\ll& \frac{e^{(1-\tfrac{\alpha}{2})\ell_2(\xi)}}{\xi^2}+% O(
\frac{1}{\xi^2}%)
\ll \xi^{-\alpha}
\eeann
as claimed.
\end{proof}

We will also require the following estimate, recording how much $f_\xi$ changes under a small perturbation in $\ell$.
\begin{lem}\label{lem:fxivar}
Let $\delta\in(0,1]$ and let $\ell,\ell'>0$ satisfy $|\ell-\ell'|<\delta$. Then for $\xi>1$, we have
\begin{equation}
|f_\xi(\ell)-f_\xi(\ell')|\ll %_{\xi}
\left\{\begin{array}{cc}
\frac{\delta}{\xi^2}, & \ell<\ell_1(\xi)-\delta,\\
\frac{\sqrt{\delta}}{\xi^2}, & \ell_1(\xi)-\delta\leq \ell\leq \ell_1(\xi)+1,\\
\frac{\delta}{\xi^2}, & \ell_1(\xi)+1\leq \ell \leq \ell_2(\xi)+1,\\
\frac{\delta}{\sinh^2(\ell)}, & \ell\geq \ell_2(\xi)+1.
\end{array}\right.
\end{equation}
\end{lem}
\begin{proof}
The derivative $f_\xi'(\ell)$ blows up at $\ell_1=\ell_1(\xi)$, but away from that point, a simple estimate using \eqref{eq:fxi'} gives
\begin{equation}\label{eq:fxi'bound}
|f_\xi'(\ell)|\ll \left\{\begin{array}{cc}
\frac{1}{\xi^2}, & \ell<\ell_1,\\
\frac{1}{\xi^2\sqrt{\delta}}, & \ell_1+\delta< \ell< \ell_1+1,\\
\frac{1}{\xi^2}, & \ell_1+1< \ell< \ell_2,\\
\frac{1}{\sinh^2(\ell)}, & \ell> \ell_2.
\end{array}\right.
\end{equation}

When $\ell\in(\ell_1-2\delta,\ell_1+2\delta)$ is close to the singular point, we use the crude bound
$$
|f_\xi(\ell)-f_\xi(\ell')|\leq |f_\xi(\ell_1)-f_\xi(\ell_1-3\delta)|+|f_\xi(\ell_1+3\delta)-f_\xi(\ell_1)|.
$$
The first term is bounded by $O(\frac{\delta}{\xi^2})$. For the second term, using \eqref{eq:fXiDef} together with the estimates
\[\cosh(\ell_1+\delta)=\sqrt{\xi^2+1}+O(\delta\xi),\quad \sinh(\ell_1+\delta)=\xi+O(\delta\xi),\]
we get that
$$
|f_\xi(\ell_1+3\delta)-f_\xi(\ell_1)|\ll \frac{\sqrt{\delta}}{\xi^2}.
$$

For the remaining cases, $|\ell-\ell_1|>2\delta$, hence $|\ell'-\ell_1|>\delta$, and we can bound $|f_\xi(\ell)-f_\xi(\ell')|\leq \delta|f_\xi'(t)|$ with $t$ between $\ell$ and $\ell'$. The result follows immediately from \eqref{eq:fxi'bound}.
\end{proof}

\newpage

%%%%%%%%%%%%%%%%%%%%%%%%%%%%%%%%%%%%%%%%%%%%%%%%%%

\section{Computing Volumes}\label{sec:vol}

In this section we asymptotically compute the volume of the region $\cR_M(Q,\xi)$ in \eqref{eq:RMdef} to prove Proposition \ref{prop:vol}.
Recall the notation in \eqref{eq:ABCdef} that $A=\cosh(\ell)$, $B=\sinh(\ell)$, and $C=2\sinh(\ell/2)$, where $\ell=\ell(M)$ as in \eqref{eq:ellM}.
We first prove the following asymptotic formula.
 \begin{prop}
 For $M\in \G, M\not\in K$, %with
 we have
 \be\label{eq:cRMvol1}
 \vol(\cR_M(Q,\xi))
 = Q^2\int_{-1}^1 \frac{|J_\xi(y)|}{\sqrt{1-y^2}}dy
 +
 O
 %_{\xi}
 \left( %{1\over Q^{2}} +
 \|M\|^{2}Q^{2/3}
 \left(1+\frac1\xi\right)
 +
 \frac{\xi^{2}}{Q^{2}}
  \right),
 \ee
as $Q\to\infty$,
 where $J_{\xi}(y)\subseteq [0,1]$ is the interval defined by
 \be\label{eq:JxiDef}
 J_{\xi}(y):=\left\{x\in [0,1]: {B\sqrt{1-y^2} \over \xi(A+By)}\leq x\leq \frac{1}{A+By}\right\}
 .
 \ee
 \end{prop}
 \begin{proof}
 We may assume that
  \be\label{eq:MtoTtoQ}
 \|M\|<Q^{2/3}
 ,
 \ee
for otherwise the estimate is trivial.
We introduce a large truncation parameter $X$ in the range
\be\label{eq:XtoQ}
%\be\label{eq:AtoX}
5\|M\|^{2}=
10A<%\frac1{10} X,
%\ee
X<Q^{2},
\ee
to be chosen later in \eqref{eq:XtoMQ}, and let
$$
\widetilde\cR_{M}(Q,\xi):=\cR_{M}(Q,\xi)\cap\{g\in G:\|g\|^{2}>X\},
$$
where   $\cR_{M}$ is the region in \eqref{eq:RMdef}.
We then clearly have
\be\label{eq:cRMtoTil}
\vol(\cR_{M}(Q,\xi))=
\vol(\widetilde\cR_{M}(Q,\xi))
+
O(X)
.
\ee

%\newpage

%Use the  $KA^+K$ coordinates \eqref{eq:KAKcoords}, a computation shows that the assumption \eqref{eq:dgNorm} forces the normalization
%\be\label{eq:dgIs}
%dg=\frac1{2\pi^{2}}d\gt\,\sinh t\, dt\ d\vf.
%\ee
Observe that in the $KA^+K$ coordinates \eqref{eq:KAKcoords},
$\widetilde\cR_{M}$
is left $K$-invariant, that is, poses no restriction on $\gt$.
In the other coordinates, we use \eqref{eq:BQ}--\eqref{eq:R2} and Lemma \ref{lem:hyp} to parametrize the region as
\bea
\label{eq:BQrest}
X<2\cosh t
&<&
Q^{2},
\\
\label{eq:R1rest}
2 (
A
 \cosh t +B \cos(\vf)
 \sinh t)
&<&Q^{2},
\\
\label{eq:R2rest}
\left|
\frac{B
\sin (\vf)}{A\sinh t+ B\cos( \vf) \cosh t   }
\right|
&<&
\tan\left(\frac{2\xi}{Q^{2}}\right)
.
\eea
%We have dropped the condition $0<|\cdots|$ in \eqref{eq:R2rest}, since the complement clearly has measure zero.
Thus, recalling our normalization \eqref{eq:dgIs}, we have
 \begin{eqnarray*}
 \vol(\widetilde\cR_M(Q,\xi))
 &=&
 \frac{1}{2\pi}\int_{-\pi}^\pi\int_0^\infty\int_{-\pi}^\pi \chi_{Q,\xi}(t,\vf)\ d\gt\,\sinh(t)dt\ d\vf
 \\
 &=&
 2
 \int_{0}^{\pi}\int_0^\infty \chi_{Q,\xi}(t,\vf)\sinh(t)dt\ d\vf
 \end{eqnarray*}
 where $\chi_{Q,\xi}$ denotes the indicator function of the set of
 $$
 (t,\vf)\in [0,\infty)\times [-\pi,\pi)
 $$
 satisfying \eqref{eq:BQrest}--\eqref{eq:R2rest}. (The condition in $\vf$ is invariant under $\vf\mapsto -\vf$ so we may multiply by 2 and integrate over $\vf\in [0,\pi)$.)

Define new coordinates given by
\be\label{eq:xyCoordinates} x=\frac{2\cosh(t)}{Q^2},\quad y=\cos(\vf).\ee
In these coordinates, the conditions \eqref{eq:BQrest}--\eqref{eq:R2rest} are equivalent to
$$
\frac{X}{Q^2}<x<1,
$$
%and conditions \eqref{eq:R1rest} and  are, respectively,
\be\label{eq:R1xy}
x\left(A + yB%\sqrt{1-\frac{4}{Q^4x^2}}
z\right)<1,
\ee
and
\be\label{eq:R2xy}
\frac{B\sqrt{1-y^2}}{x\left|A % \sqrt{1-\frac{4}{Q^4x^2}}
z
+ y B\right|}<\frac{Q^2}{2}\tan\left(\frac{2\xi}{Q^2}\right)
,
\ee
where
$$
z:=\sqrt{1-\frac{4}{Q^4x^2}}\in(0,1).
$$

For any $y\in [-1,1]$ let $\widetilde J_{Q,\xi}(y)$ denote the set of all $x\in [\frac{X}{Q^2},1]$ satisfying \eqref{eq:R1xy} and \eqref{eq:R2xy}; then
\be\label{eq:volMid}
\vol(\widetilde\cR_M(Q,\xi))=%\frac
{Q^2}%{\pi}
\int_{-1}^1|\widetilde J_{Q,\xi}(y)|\frac{dy}{\sqrt{1-y^2}}.
\ee

%We need to truncate this region a bit to keep $z$ near $1$.

%\newpage

%To this end, we introduce a large parameter
%and
%let
%$$
%\tilde{J}_{Q,\xi}(y):=J_{Q,\xi}(y)\cap \left[\frac{X}{Q^2},1\right]
%,
%$$
%and note that
%so that
%\be\label{eq:tilJtoJ}
%|\tilde{J}_{Q,\xi}(y)|=|J_{Q,\xi}(y)|+O\left(\frac{X}{Q^2}\right).
%\ee
Now for $x\in \widetilde{J}_{Q,\xi}(y)$, we estimate
\be\label{eq:4X2}
1-
%\sqrt{1-\frac{4}{Q^4x^2}}
z\le\frac{4}{X^2},
\ee
whence
 \eqref{eq:R1xy}  can be replaced by
\bea\nonumber
x
&<&
\frac1{A + yBz%\sqrt{\cdots}
}
=
\frac1{A + yB}
+
%\left(
\frac{yB(1-z%\sqrt{\cdots}
)}{(A + yBz%\sqrt{\cdots}
)(A+yB)
}
%\right)
\\
\label{eq:xUp}
&=&
\frac1{A + yB}
+
O
\left(
{\|M\|^{6}
\over X^{2}
}
\right)
,
\eea
where we estimated
$$
A+Byz
%\sqrt{\cdots}
\geq{A-B}\geq \frac1{2A}=\|M\|^{-2}.
$$

Next to replace  \eqref{eq:R2xy}, we
%assume that
%\newpage
use \eqref{eq:XtoQ}
and
\eqref{eq:4X2}
to
%from which we can
estimate
$$
Az
%\sqrt{\cdots}
+By
\ge
A-B
-
A(1-
%\sqrt{\cdots}
z)
\ge
\frac1{2A}
-
\frac {4A}{X^{2}}
\gg
\frac1A
%,
.
$$
%where we used
In particular, $Az+By>0$, so we do not need the absolute values in \eqref{eq:R2xy}.
Now start with  the estimate
$$
\frac{Q^{2}}{2}\tan{\frac{2\xi}{Q^{2}}}=\xi+O%_\xi
\left(\frac{\xi^{3}}{Q^4}\right),
$$
for the right hand side of \eqref{eq:R2xy}, multiply both sides by $x$, divide by $\xi$,
and use
$$
\frac{\sqrt{1-y^2}}{A+By}\leq 1
,
$$
to arrive at
%we can replace \eqref{eq:R2xy} by
\bea\nonumber
%%%%%%%%%%%%%%%%%%%%
\begin{comment}
x
\frac1\xi
\frac{Q^2}{2}\tan\left(\frac{2\xi}{Q^2}\right)
&=&
\end{comment}
%%%%%%%%%%%%%%%%%%%%
x
&>&
\frac{B\sqrt{1-y^2}}{\xi\left(Az
%\sqrt{1-\frac{4}{Q^4x^2}}
+ y B\right)}
+
O%_{\xi}
\left(\frac{\xi^{2}}{Q^{4}}\right)
\\
&=&
\begin{comment}
%%%%%%%%%%%%%%%%%%%%%%%%%%%%%%%%%%%%%%%%%%%%%%%%%%%%%%
\frac{B\sqrt{1-y^2}}{\xi\left(A + y B\right)}
+
{B\sqrt{1-y^{2}}\over \xi}
\left(
\frac{A (1-\sqrt{1-\frac{4}{Q^4x^2}}) }
{\left(A\sqrt{1-\frac{4}{Q^4x^2}} + y B\right)\left(A + y B\right)}
\right)
\\
&=&
\frac{B\sqrt{1-y^2}}{\xi\left(A + y B\right)}
+
O
\left(
{\|M\|^{4}\over X^{2}}
\frac1
{A\sqrt{1-\frac{4}{Q^4x^2}} + y B}
\right)
\\
&=&
%%%%%%%%%%%%%%%%%%%%%%%%%%%%%%%%%%%%%%%%%
\end{comment}
\label{eq:xLow}
\frac{B\sqrt{1-y^2}}{\xi\left(A + y B\right)}
+
O
%_{\xi}
\left(
\frac{\xi^{2}}{Q^{4}}
+
 \frac{\|M\|^{6}
}{\xi X^{2}}
\right)
.
\eea

Finally, combining % \eqref{eq:tilJtoJ},
\eqref{eq:xUp} and \eqref{eq:xLow}, we obtain
\[
|\widetilde J_{Q,\xi}(y)|
=
|J_{\xi}(y)|
+
O%_{\xi}
\left(
\frac{X}{Q^2}
+
\frac{\xi^{2}}{Q^{4}}
+
 \frac{\|M\|^{6}}{\xi X^{2}}
\right)
,\]
where $J_{\xi}(y)$ is as defined in \eqref{eq:JxiDef}.
%Taking the optimal value
Now to balance the error terms, we take
\be\label{eq:XtoMQ}
X=\|M\|^{2}Q^{2/3},
\ee
 whence
%\eqref{eq:AtoX} is immediately satisfied (for $Q$ large), and
\eqref{eq:XtoQ} is satisfied by
\eqref{eq:MtoTtoQ} and taking $Q$ large.
%Lastly, since $M\in\G$ and $M\not\in K$, we have that $\|M\|$ is at least the injectivity radius of $\G$; hence $\|M\|\gg1$, and the $1/Q^{4}$ error term may be dropped.
Inserting these estimates into
 \eqref{eq:volMid} and \eqref{eq:cRMtoTil} gives \eqref{eq:cRMvol1}, as claimed.
 \end{proof}

To prove Proposition \ref{prop:vol}, it remains to analyze the integral in the main term of \eqref{eq:cRMvol1}.
(Note that $\|M\|=\sqrt{2\cosh\ell(M)}\ge\sqrt2$, so the $\xi^{2}/Q^{2}$ error
%by \eqref{eq:MtoTtoQ}, we have $\|M\|^{4}Q^{-4}< \|M\|^{2}Q^{-8/3}$,  so this
term in \eqref{eq:cRMvol1} may be dropped.)

The length $|J_\xi(y)|$ is given explicitly by
\[
|J_\xi(y)|
=
\threecase
{1-{B\sqrt{1-y^2} \over \xi(A+By)}, }{ if $y\in I_1(\xi)$,}
{\frac{1}{A+By}-{B\sqrt{1-y^2} \over \xi(A+By)} ,}{ if  $y\in I_2(\xi)$,}
{0,} { if $ y\in I_3(\xi)$,}
\]
where
$$I_1(\xi)=\left\{y\in [-1,1]:{B\sqrt{1-y^2} \over \xi(A+By)} <1\le\frac{1}{A+By}\right\},$$
$$I_2(\xi)=\left\{y\in [-1,1]:{B\sqrt{1-y^2} \over \xi(A+By)} <\frac{1}{A+By}\le 1\right\},$$
and $I_3(\xi)=[-1,1]\setminus (I_1\cup I_2)$.

Solving these inequalities, we get an explicit description for the intervals $I_1(\xi)$ and $I_2(\xi)$.

\begin{lem}
Let $y=\lambda_{\pm}(\xi)$ denote the roots of the quadratic equation
\be\label{eq:eqI}
B^2(\xi^2+1)y^2+2AB\xi^2y+A^2\xi^2-B^2=0,
\ee
and
if $\xi\le B$, then %we
set
%let
$y=\alpha(\xi)$ to be the %positive
non-negative
solution to $B\sqrt{1-y^2}=\xi$.
We then have that
\be \label{eq:I1} I_1(\xi)=
\threecase
{[-1,\frac{1-A}{B}]   ,} { if $B<\xi,$}
{[-1,\lambda_{-}(\xi))\cup (\lambda_{+}(\xi),\frac{1-A}{B}] ,}{ if  $C<\xi\le B$,}
{[-1,\lambda_{-}(\xi)),}{ if $\xi\le C
%>\xi
 $,}
\ee
and
\be \label{eq:I2}
I_2(\xi)=
\threecase
{[\frac{1-A}{B},1] ,}{ if $ B< \xi,$}
{[\frac{1-A}{B},-\alpha(\xi))\cup (\alpha(\xi),1],}{ if $C<\xi\le B$,}
{(\alpha(\xi),1] ,}{ if $\xi\le C$.}
\ee
\end{lem}

\pf
We analyze  $I_{1}(\xi)$, the analysis of $I_{2}(\xi)$ being similar.
The equation \eqref{eq:eqI} corresponds to the boundary condition $\frac{B\sqrt{1-y^2}}{\xi(A+By)}=1$.
This quadratic equation in $y$ is asymptotically positive, and has discriminant
$$
4B^{2}(B^{2}-\xi^{2})
,
$$
so the roots $\gl_{\pm}(\xi)$ are only real if $B\ge\xi$. Thus if $B<\xi$, we only have the restriction $y
%<
\leq
{1-A\over B}$, giving the top line of \eqref{eq:I1}.

When $\gl_{+}(\xi)<{1-A\over B}$, we obtain the middle line of \eqref{eq:I1}, with the boundary condition $\gl_{+}(\xi)={1-A\over B}$ being equivalent to
$$
\xi^{2}=2A-2=C^{2}.
$$
The bottom line of \eqref{eq:I1} is then immediate.
\epf

We may now complete the
\begin{proof}[Proof of Proposition \ref{prop:vol}]
We define
$$
F_M(\xi):=\int_{-1}^1|J_{\xi}(y)|\frac{dy}{\sqrt{1-y^2}},
$$
so that the main term of \eqref{eq:cRMvol1} is ${Q^{2}} F_{M}(\xi)$. It is easy to see that $F_{M}(\xi)\to0 $ as $\xi\to0$, so it remains now to prove that
$$
{d\over d\xi}F_{M}(\xi)=f_{\xi}(\ell(M)).
$$

We  write explicitly
\begin{eqnarray*}
F(\xi)&=&\int_{I_1(\xi)}\left(1-{B\sqrt{1-y^2} \over \xi(A+By)}\right)\frac{dy}{\sqrt{1-y^2}}\\&&+\int_{I_2(\xi)}\left(\frac{1}{A+By}-{B\sqrt{1-y^2} \over \xi(A+By)}\right)\frac{dy}{\sqrt{1-y^2}}
,\\
\end{eqnarray*}
and compute the derivative in the three cases separately.

%First consider the case when
{\bf Case I: $\xi< C$.}
Here  we have
\begin{eqnarray*}
F(\xi)&=&\int_{-1}^{\lambda_-(\xi)}\left(1-{B\sqrt{1-y^2} \over \xi(A+By)}\right)\frac{dy}{\sqrt{1-y^2}}\\&&+\int_{\alpha(\xi)}^1\left(\frac{1}{A+By}-{B\sqrt{1-y^2} \over \xi(A+By)}\right)\frac{dy}{\sqrt{1-y^2}},
\end{eqnarray*}
and hence
\begin{eqnarray*}
F'(\xi)&=&\left(1-{B\sqrt{1-(\lambda_{-}(\xi))^2} \over \xi(A+B\lambda_{-}(\xi))}\right)\frac{\lambda_{-}'(\xi)}{\sqrt{1-(\lambda_{-}(\xi))^2}}\\&&-\left(\frac{1}{A+B\alpha(\xi)}-{B\sqrt{1-\alpha(\xi)^2} \over \xi(A+B\alpha(\xi))}\right)\frac{\alpha'(\xi)}{\sqrt{1-\alpha(\xi)^2}}\\
&&+\frac{1}{\xi^2}\left(\ln(A+B\lambda_{-}(\xi))-\ln(A-B)-\ln(A+B\alpha(\xi))+\ln(A+B)\right)
.
\end{eqnarray*}
Note that $\lambda_{\pm}(\xi)$ and $\alpha(\xi)$ are the precise points where $\frac{B\sqrt{1-(\lambda_{\pm}(\xi))^2}}{\xi(A+B\lambda_{\pm}(\xi))}=1$ and $\frac{1}{A+B\alpha(\xi)}=\frac{B\sqrt{1-\alpha(\xi)^2}}{\xi(A+B\alpha(\xi))}$, respectively. Thus the first two terms vanish, giving
\begin{eqnarray*}
F'(\xi)%&=&\frac{1}{\pi\xi^2}\left(\ln(A+B\lambda_{-}(\xi))-\ln(A-B)-\ln(A+B\alpha(\xi))+\ln(A+B)\right)\\
&=&\frac{2}{\xi^2}\ln(\frac{A+B}{A+\sqrt{B^2-\xi^2}})
,
\end{eqnarray*}
as claimed.

%Next, when
{\bf Case II: $C<\xi< B$.} Here we have
\begin{eqnarray*}
F(\xi)&=&\int_{-1}^{\lambda_-(\xi)}\left(1-{B\sqrt{1-y^2} \over \xi(A+By)}\right)\frac{dy}{\sqrt{1-y^2}}\\
&&+\int_{\lambda_+(\xi)}^{\frac{1-A}{B}}\left(1-{B\sqrt{1-y^2} \over \xi(A+By)}\right)\frac{dy}{\sqrt{1-y^2}}\\
&&+\int_{\frac{1-A}{B}}^{-\alpha(\xi)}\left(\frac{1}{A+By}-{B\sqrt{1-y^2} \over \xi(A+By)}\right)\frac{dy}{\sqrt{1-y^2}}\\
&&+\int_{\alpha(\xi)}^1\left(\frac{1}{A+By}-{B\sqrt{1-y^2} \over \xi(A+By)}\right)\frac{dy}{\sqrt{1-y^2}}
.
\end{eqnarray*}
When taking derivatives, the contribution of the end points cancel out as before, and we obtain
\begin{eqnarray*}
 F'(\xi)&=&\frac{1}{\xi^2}\left(\ln(A+B\lambda_-(\xi))-\ln(A-B)-\ln(A+B\lambda_+(\xi))\right.\\
&&\left.+\ln(A-B\alpha(\xi))-\ln(A+B\alpha(\xi))+\ln(A+B)\right)\\
&=& \frac{2}{\xi^2}\ln(\frac{(A+B)(1+\xi^2)}{(A+\sqrt{B^2-\xi^2})^2})
,
\end{eqnarray*}
as desired.

{\bf Case III:
%Finally, when
 $B<\xi$.} Now we have
\begin{eqnarray*}
F(\xi)&=&\int_{-1}^{\frac{1-A}{B}}\left(1-{B\sqrt{1-y^2} \over \xi(A+By)}\right)\frac{dy}{\sqrt{1-y^2}}\\
&&+\int_{\frac{1-A}{B}}^1\left(\frac{1}{A+By}-{B\sqrt{1-y^2} \over \xi(A+By)}\right)\frac{dy}{\sqrt{1-y^2}}
.
\end{eqnarray*}
Here the only dependence on $\xi$ is in the integral and we get
\begin{eqnarray*}
F'(\xi)&=&\frac{1}{\xi^2}(\ln(A+B)-\ln(A-B))= \frac{2\ell}{\xi^2}.
\end{eqnarray*}

We have now verified that the derivative  $F'$ agrees with $f_{\xi}(\ell)$ away from the potential points of discontinuity, but $f_{\xi}$ is continuous (cf. Remark \ref{rmk:fxiCont}) so we are done.
This completes the proof of Proposition \ref{prop:vol}.
\end{proof}

\newpage

\section{Relating the Counts to Volumes}\label{sec:cM}

The purpose of this section is to prove Proposition \ref{prop:cM}.
While the idea of the proof is more-or-less standard, the region $\cR_{M}(Q,\xi)$ in \eqref{eq:RMdef} is not exactly well-rounded, resulting in a number of technical obstructions which must be overcome before the method works.

First some preliminaries.
We may certainly assume that
\be\label{eq:MbndQ}
\|M\|<%\frac1{10}
Q^{(1-2\gT)/16},
\ee
or else Proposition \ref{prop:cM} is trivial.

Next, to address the issue of well-roundedness, we introduce a truncation parameter
\be\label{eq:Xtrunc}
%\|M\|<
1<X<{Q\over 20\|M\|}
,
\ee
and define
\be \label{eq:truncR}
\cR_M(Q,\xi;X):=\left\{g\in \cR_M(Q,\xi): \|g\|>%\max\left(
\tfrac{Q}{X}%,10Y\|M\|\right)
\right\}.
\ee
The volume of the complement is clearly $O\left(\tfrac{Q^{2}}{X^{2}}%+Y^{2}\|M\|^{2}+1
\right)$.

Let $\delta>0$ be some small parameter and let $\delta_1$ be related to $\gd$ by \eqref{eq:gdTogd1}. %=\sqrt{2\cosh(\delta)}$.
Recall that the regions
 $B_{\gd_{1}}$  and $D_{\gd}$ are defined in  \eqref{eq:Ddel} and \eqref{eq:Ddel1}, respectively.
Let $\cR^+_{M}$ be the $B_{\delta_1}\times D_\delta$-thickening of $\cR_{M}$,  that is,
\be \label{eq:Rplus}
\cR^+_{M}(Q,\xi;X%,Y
):=
B_{\gd_{1}}\cdot \cR_{M}(Q,\xi;X%,Y
)\cdot D_{\gd}
%\{h_1 g h_2: g\in \cR_{M}(Q,\xi;X,Y),\; (h_1,h_2)\in B_{\delta_1}\times D_\delta\}
.
\ee
Similarly, let $\cR^-_{M}$ be the set whose thickening is in $\cR_{M}$,
\be \label{eq:Rminus}
\cR^-_{M}(Q,\xi;X%,Y
):=
\bigcap_{(h_{1},h_{2})\in B_{\gd_{1}}\times D_{\gd}}
h_{1}\cdot \cR_{M}(Q,\xi;X%,Y
) \cdot h_{2}
%\{g: h_1 g h_2\in
%\cR_{M}(Q,\xi;X,Y)
%,\;\forall (h_1,h_2)\in B_{\delta_1}\times D_\delta\}.
\ee
With the region $\cR_{M}$ slightly truncated, we can now prove the following well-roundedness statement.

\begin{lem}\label{lem:round}
With $\xi>0$ fixed, there
 %is a
are
 constants $c,c'>0$, depending on $\xi$ and $\G$, such that for
\be\label{eq:delToM}
\delta
%=o(
<
c'
\|M\|^{-2},
\ee
we have %(with $Y=2$)
\be\label{eq:cRplus}
\cR^+_M(Q,\xi;X)\subseteq  \cR_M(Q(1+c\delta \|M\|^2),\xi(1+c\delta X^2\|M\|^4); 2X),
\ee
and
\be\label{eq:cRmin}
\cR_M(Q,\xi;X)\subseteq  \cR^{-}_M(Q(1+c\delta \|M\|^2),\xi(1+c\delta X^2\|M\|^4);2X).
\ee
\end{lem}
\begin{proof}
Let $g_1\in \cR^+_M(Q,\xi;X)$; then $g_1=h_1^{-1}gh_2$ with $g\in\cR_M(Q,\xi;X)$,  $h_1\in B_{\delta_1}$ and $h_2\in D_\delta\subset B_{\gd_{1}}$.
The assumptions $\|g\|>Q/X$ and \eqref{eq:Xtrunc} ensure that $\|g\|>20\|M\|$, so we are in position to use Lemmata \ref{lem:pert} and \ref{lem:pert2}.
Recall that $c\asymp1$ is a constant which can change from line to line, or even in the same line.
Applying  \eqref{eq:gMch} gives
$$
\|g_1M\|^2<\|gM\|^2+c\gd\|g\|^{2}\|M\|^2< Q^2(1+c\delta \|M\|^2)^2,
$$
so the thickening replaces $Q$ by $Q(1+c\gd\|M\|^{2})$.

By \eqref{eq:dnormR} and \eqref{eq:dnormL}, we have,
$$
\|g_1\|^2= \|g\|^2(1+O(\delta))^{2},
$$
and hence \eqref{eq:delToM} gives
$$
\frac{Q^2(1+c\gd\|M\|^{2})^{2}}{(2X)^{2}}< \frac{Q^2(1-c\gd)^{2}}{X^{2}}< \|g_1\|^2<  Q^{2}(1+c\gd)^{2}<  Q^{2}(1+c\gd\|M\|^{2})^{2}.
$$
% for some $0<c\ll1$.
Thus in thickening we may replace $X$ by $2X$.

Lastly, \eqref{eq:Xtrunc} and \eqref{eq:delToM} ensure that
$$
{\gd\|M\|^{4}\over \|g\|^{2}}
<c
%=
%o
%\left(
%1
%\right)
.
$$
Then by
%whence
\eqref{eq:tanCh},
together with
 Lemma \ref{lem:hyp} (note that $\|g\|>20\|M\|$ by \eqref{eq:Xtrunc}--\eqref{eq:truncR}),
%gives
we have
\beann
|\theta_{g_1}-\theta_{g_1M}|
&\le&
|\theta_g-\theta_{gM}|+c\delta \frac{\|M\|^4X^{2}}{Q^2}
<
\frac
{
2\xi
(1
+
c_{\xi}\,\delta
\|M\|^4X^{2}
)
}
{Q^{2}}
\\
&<&
\frac
{
2\xi
%(1+c\gd\|M\|^{2})^{2}
}
{Q^{2}(1+c\gd\|M\|^{2})^{2}}
(1
+
c%_{\xi}
\,\delta
\|M\|^4X^{2}
)
.
\eeann

This proves \eqref{eq:cRplus}, and the proof of \eqref{eq:cRmin} is similar.
\end{proof}

Now we follow a standard procedure to compute the cardinality of $\G\cap\cR_{M}(Q,\xi;X)$.

\begin{prop}\label{prop:count}
Let $\xi>0$ be fixed. Recall from \eqref{eq:gl1} that $\gT\in(0,\foh)$ is a spectral gap for $\G$,
and from  \eqref{eq:VGdef} that  $V_{\G}$ is the co-volume of $\G$.
Then for any $M\in \G,$ $M\not\in K$, assuming \eqref{eq:Xtrunc} we have
\bea
\label{eq:GRMXis}
&&
\hskip-.5in
\left|
\#\G\cap\cR_{M}(Q,\xi;X)
-
{\vol(\cR_{M}(Q,\xi))\over V_{\G}}
\right|
\\
\nonumber
&&
\ll_{\xi}%\left(
{Q^{2}\over X^{2}}
+
Q^{(9+2\gT)/5}
X^{8/5}
\|M\|^{16/5}
+
\|M\|^{2}Q^{2/3}
%\right)
,
\eea
as $Q\to\infty$.
\end{prop}

\pf[Sketch of the proof]
%We will prove a lower bound, the upper bound being the same.
Let $\gd>0$ be small enough that \eqref{eq:delToM} is satisfied.
Let $\psi_1=\psi_{1,\gd}$ be a spherical $\gd$-bump function about the origin in $G$, that is, $\psi_{1}$ is smooth, non-negative,
 %$\psi_1\ge0$,
 $\int_{G}\psi_1=1$, $\psi_1(kgk')=\psi_{1}(g)$ and $\supp\psi_{1}\subset B_{\delta_1}%\subset G
 $.
Also let $\psi_2=\psi_{2,\gd}$ denote a (non-spherical) $\delta$-bump function supported on $D_\delta$.

For $j=1,2$, let
$$
\Psi_j(g)%=\Psi_{j}(g)
:=\sum_{\g\in\G}\psi_j(g\g),
$$
so that $\Psi_1
%,\Psi_{2,0}
\in L^2(K\bk G/\G)$ and $\Psi_2\in L^2( G/\G)$.
%Assuming $\delta$ is smaller then the injectivity radius of $\Gamma$ we have that
%\[\|\Psi_j\|^2=\int_G|\psi_j|^2dg,\]
%and w
We can choose the bump functions so that
\be\label{eq:L2norms}
\|\Psi_1\|\asymp
\frac1{\vol(B_{\gd_{1}})^{1/2}}
\asymp
 \frac{1}{\delta},
 \quad
%\quad \|\Psi_2\|\asymp \frac{1}{\delta^2},
\mbox{ and }
\quad
 \cS\Psi_{2}\asymp
 \frac1\gd\cdot
 \frac1{\vol(D_{\gd})^{1/2}}
 \asymp
  \frac{1}{\delta^{3}}
  .
\ee
Here $\cS$ is a first-order Sobolev norm, defined as follows. Fix a basis $X_{1},X_{2},X_{3}$  for the Lie algebra $\fg=\fs\fl_{2}(\R)$; then $\cS\Psi=\max\limits_{j=1,2,3}\|X_{j}.\Psi\|$.

Given $c>0$ from Lemma \ref{lem:round} with $c'$ small enough, let
\be\label{eq:gd2gd3}
\gd_{2}:=1+c\gd\|M\|^{2}\asymp1,\qquad
\text{and}\qquad
\gd_{3}:=1+c\gd X^{2}\|M\|^{4}\ll X^{2}\|M\|^{2}.
\ee
Let $\cF=\cF_{Q,\xi,X,M}\in L^2(K\bk G/\G\times G/\G)$ be defined by
$$
\cF(g,h):=\sum_{\g\in\G}\bo_{ \cR_{M}(Q\gd_{2},\xi \gd_{3};2X)}(g\g h^{-1}).
$$
We will prove an upper bound on the cardinality of $\G\cap \cR_{M}(Q,\xi;X)$, the lower bound being similar.

Using \eqref{eq:cRmin}, an easy calculation shows that
\be\label{eq:GupBnd}
\#\G\cap \cR_{M}(Q,\xi;X)\
\le\
\#\G\cap \cR_{M}^{-}(Q\gd_{2},\xi\gd_{3};2X)\
\le\
\<\cF,\Psi_1\otimes\Psi_2\>
,
\ee
and that
\be\label{eq:unfold}
\<\cF,\Psi_1\otimes\Psi_2\>
=
\int_{\cR_{M}(Q\gd_{2},\xi\gd_{3};2X)}
\<\pi(g%^{-1}
).\Psi_2,\Psi_1\>
dg
,
\ee
where $\pi$ is the left-regular representation on $G$.

To estimate the last integral, decompose each % of the
function as
$$\Psi_j=\frac{1}{
%\vol(\G\bk \bH)
%\sqrt{
V_{\G}}+\Psi_j^\bot,
$$
where $\Psi_j^\bot$ is orthogonal to
constants.
Recall the well-known decay of matrix coefficients
%Applying the well-known decay of matrix coefficients
(see, e.g., \cite{Warner1972, CowlingHaagerupHowe1988, Shalom1999}, and in particular, \cite[\S9.1.2]{Venkatesh2010}):
for mean-zero functions $F_{1},F_{2}\in L^{2}_{0}(G/\G)$, we have
\be\label{eq:MtrxC}
\<\pi(g).F_{1},F_{2}\>
\ll
\|g\|^{-1+2\gT}\
\cS F_{1}\cdot
\cS F_{2}
,
\ee
where $\gT$ is a spectral gap for $\G$. We are being a bit crude here with the $F_{j}$ dependence in the error; as we are not trying to optimize exponents, we opt for a cleaner statement than best known. That said, when an $F_{j}$ is $K$-fixed, then its Sobolev norm can be replaced by its $L^{2}$-norm.

Applying \eqref{eq:MtrxC}, we have
$$
\<\pi(g%^{-1}
).\Psi_2,\Psi_1\>
=
\frac1{V_{\G}}
+
O(\|g\|^{-1+2\gT}\|\Psi_{1}\|\cS\Psi_{2})
=
\frac1{V_{\G}}
+
O\left(\|g\|^{-1+2\gT}
\frac1{\gd^{4}}
\right)
,
$$
by \eqref{eq:L2norms}.
Inserting
 %\eqref{eq:MtrxC}
 this
 into \eqref{eq:unfold}
and
 \eqref{eq:GupBnd} gives
\bea
\nonumber
\#\G\cap \cR_{M}(Q,\xi;X)
&\le&
{\vol(\cR_{M}(Q\gd_{2},\xi\gd_{3};2X))
\over
V_{\G}
}
+
O\left(
\frac1{\gd^{4}}
Q^{1+2\gT}
\right)
\\
\nonumber
&=&
{\vol(\cR_{M}(Q\gd_{2},\xi\gd_{3}))
\over
V_{\G}
}
+
O\left(
{Q^{2}\over X^{2}}
+
\frac1{\gd^{4}}
Q^{1+2\gT}
\right)
,
\\
\label{eq:GcapRMX}
\eea
where % in the first error term,
we crudely estimated $\cR_{M}(Q\gd_{2},\xi\gd_{3};2X)\subset B_{2Q}$.

Using the estimate
%\newpage
%
 in Proposition \ref{prop:vol} for the volumes, with the uniform in $\xi$ error terms given in \eqref{eq:cRMvol1}, we see that
\bea
\nonumber
&&
\hskip-.5in
|\vol(\cR_M(Q\delta_2,\xi\delta_3))
-
 \vol(\cR_M(Q,\xi))|
 \\
 \nonumber
 &
 %\le
 =
 &
 Q^{2}\gd_{2}^2\int_{\xi}^{\xi\gd_{3}}f_{\gz}(\ell(M))d\gz
 +
 Q^{2}(\gd_{2}^2-1)\int_{0}^{\xi}f_{\gz}(\ell(M))d\gz
 \\
 \nonumber
&&
\quad
+
O_{\xi}(\|M\|^{2}Q^{2/3}
)
+
O(
\|M\|^{2}(Q\gd_{2})^{2/3}(1+1/(\xi\gd_{3}))
+(\xi\gd_{3}/Q)^{2}
)
\\
\nonumber
 &\ll_\xi&
 Q^{2}
(\gd_{3}-1)
 +
 Q^{2}
 (\gd_2-1)
+
\|M\|^{2}Q^{2/3}
+
{X^{2}\|M\|^{2}\over Q^{2}}
\\
\label{eq:volRMX}
&\ll&
 Q^{2}\gd X^{2}\|M\|^{4}
+
\|M\|^{2}Q^{2/3}
,
\eea
where we used \eqref{eq:delToM}, \eqref{eq:gd2gd3}, and \eqref{eq:Xtrunc}.
%, and \eqref{eq:fxiDecay}.%, and \eqref{eq:Xtrunc}

Combining \eqref{eq:GcapRMX} with \eqref{eq:volRMX}, we choose the  optimal value
$$
\gd
=
{
1
\over
Q^{(1-2\gT)/5}
X^{2/5}\|M\|^{4/5}
}
.
$$
It is easy to see from  \eqref{eq:MbndQ} % and \eqref{eq:Xtrunc}
that \eqref{eq:delToM} is satisfied.

We thus obtain
\beann
\#\G\cap \cR_{M}(Q,\xi;X)
&\le&
{\vol(\cR_{M}(Q,\xi))
\over
V_{\G}
}
\\
&&
+
O_{\xi}
\left(
{Q^{2}\over X^{2}}
+
\|M\|^{2}Q^{2/3}
+
Q^{(9+2\gT)/5}
X^{8/5}
\|M\|^{16/5}
\right)
.
\eeann
The lower bound is proved similarly, concluding the proof of
 \eqref{eq:GRMXis}.
\epf

We are finally in position to give a
\pf[Proof of Proposition \ref{prop:cM}]
We easily see that
\beann
\#\G\cap\cR_{M}(Q,\xi)
&=&
\#\G\cap\cR_{M}(Q,\xi;X)
+O\left({Q^{2}\over X^{2}}\right)
.
\eeann
Combined with \eqref{eq:GRMXis}, we choose the optimal value
$$
X=
{Q^{(1-2\gT)/18}
\over
\|M\|^{8/9}
}
.
$$
Then \eqref{eq:Xtrunc} is satisfied by \eqref{eq:MbndQ}.
We can also use \eqref{eq:MbndQ} to bound
$$\|M\|^2Q^{2/3}\leq \|M\|^{16/9}Q^{(17+2\gT)/9},$$
and the claim follows immediately.
\epf

\begin{rem}
For the generalization described in Remark \ref{rmk:Jens} we need to
replace the regions $\cR_M(Q,\xi)$ by
$$\cR^{(\cI)}_N(Q,\xi)=\{g\in \cR_M(Q,\xi)|\theta_g\in \cI\},$$
and similarly $\cR_M(Q,\xi,X)$ by $\cR_M^{(\cI)}(Q,\xi,X)$; when doing this the volumes change by a factor of $|\cI|$.
These new regions are no longer left $K$-invariant, which requires a few small modifications to the proofs. In particular, in the smoothing process, \eqref{eq:Rminus} and \eqref{eq:Rplus}, we need to replace $B_{\delta_1}$ by $D_\delta$, and in Lemma \ref{lem:round} we also need to enlarge the interval $\cI$ on the left hand side by $O(\delta)$. Next, in the proof of Proposition \ref{prop:count} we need both $\delta$-bump functions to be non-spherical, resulting in a slightly worse power saving for the error term. After these modification the rest of the proof follows without a change.
\end{rem}

\newpage

\section{Bounding the Error $\cE$}\label{sec:cE}

Recall the notation in \S\ref{sec:intro}.
In this section, we prove Proposition \ref{prop:cE}, % by
estimating the ``error'' term
$
\cE_{Q,T}(\xi)
$.
Recall that this is half the cardinality of the set
$$
\cS=
\left\{
(\gamma,M)\in\G^{2}: \|\gamma\|,\|\gamma M\|\leq Q,\; \|M\|\ge T, \; |\theta_{\gamma}-\theta_{\gamma M}|<\frac{2\xi}{Q^2}
\right\}
.
$$
It will be more convenient to return to $\gamma'=\gamma M$, renaming
\[
\cS=
\left\{
(\gamma,\gamma')\in\G^{2}:
\|\gamma\|,\|\gamma'\|\leq Q,\; |\theta_{\gamma}-\theta_{\gamma'}|<\frac{2\xi}{Q^2},\; \|\gamma^{-1}\gamma'\|\ge T
\right\}
.
\]

We first prove the following
\begin{lem}
For $10< Q_{1}<Q$, and $\gt'\in[-\pi,\pi)$ fixed,
\be\label{eq:Q1bnd}
\#
\left\{
\g\in\G:
Q_{1}\le \|\g\|<2 Q_{1}
,
|\gt_{\g}-\gt'|<{2\xi\over Q^{2}}
\right\}
\ll_{\xi}
1
.
\ee
\end{lem}
Recall that the implied constant above, as throughout, may depend on $\G$ without further notice.
\pf
Fix  $\delta=\gd(\G)\asymp1$ sufficiently small
%, in particular smaller than half the injectivity radius,
%so
that
$$
(KA_{2\gd}K)\cap\G
%\quad \subset \quad
%B_{\delta_{1}^{2}}\cap \Gamma\quad
=
%\quad
\{1\}.
$$
Letting $\gd_{1}$ be related to $\gd$ by \eqref{eq:gdTogd1}, recall the region $B_{\gd_{1}}\subset G$ defined in \eqref{eq:Ddel}.

Writing $\cL$ for
the left hand side of \eqref{eq:Q1bnd}, we thicken each $\g$ on the right by $B_{\gd_{1}}$, giving
\beann
\cL
&=&
\frac1{\vol( B_{\gd_{1}})}
\sum_{\g\in\G\atop
Q_{1}\le \|\g\|<2 Q_{1},
|\gt_{\g}-\gt'|<\frac{2\xi}{Q^{2}}}
{\vol(\g B_{\gd_{1}})}
\\
&\le &
\frac1{\vol( B_{\gd_{1}})}
\vol\left\{g\in G:
\|g\|\le3 Q_{1},
|\gt_{g}-\gt'|\le \frac{c\gd+2\xi}{Q_{1}^{2}}
\right\}
,
\eeann
where we used \eqref{eq:dnormR} and \eqref{eq:dgtR}.
It is elementary to compute that this volume is $\ll_{\xi} 1$, giving the claim (since $\gd\asymp1$).
\epf

We can now give a
\pf[Proof of Proposition \ref{prop:cE}]\

Let $(\g,\g')\in\cS$ and
write $\g=k_{\theta} a_t k$ and $\g'=k_{\theta'}a_{t'} k'$, so that, as in \eqref{eq:R1}, we have
\beann
\|\g^{-1}\g'\|^2
&=&
\|a_{-t}k_{\theta'-\theta}a_{t'}\|^2
\\
&=&
2(\cosh(t-t')\cos^2\left({\theta-\theta'\over2}\right)+\cosh(t+t')\sin^2\left({\theta-\theta'\over2}\right)
.
\eeann
Since $(\g,\g')\in\cS$, we have
$|\theta-\theta'|\leq \frac{2\xi}{Q^2}$
and
$2\cosh t,2\cosh t'\leq Q^{2}$,
whence
$$
\|\gamma^{-1}\gamma'\|^{2}
=
2\cosh(t-t')+O_{\xi}(1%Q^{-2}
)
.
$$
Then for $T\gg_{\xi}1$, the condition
$\|\gamma^{-1}\gamma'\|\geq T$
%is implied by
implies
$$
2\cosh(t-t')>\foh T^{2}.
$$
Assuming $t'\le t$, we can relax this even further to $t'<2\log (Q/T)+c$ and $t<2\log Q$.

Hence $\cE_{Q,T}(\xi)=\foh\#\cS\ll \#\cS'$, where
\beann
\cS'
&:=&
\left\{
(\g,\g')\in\G^{2}:
\|\gamma\|\leq Q,\;
\|\gamma'\|\ll\frac QT,\;
|\theta_{\gamma}-\theta_{\gamma'}|<\frac{2\xi}{Q^2}
\right\}
.
\eeann
We sum $\g'$ on the outside, and break the $\g$ sum dyadically:
\beann
\#\cS'
&\ll&
\sum_{\g'\in\G\atop\|\g'\|\ll \frac QT}
\left(
1
+
\sum_{10<Q_{1}<Q\atop \text{dyadic}}
\#\left\{
\g\in\G:
\|\gamma\|\asymp Q_{1},\;
|\theta_{\gamma}-\theta_{\gamma'}|<\frac{2\xi}{Q^2}
\right\}
\right)
\\
&\ll_{\xi}&
\log Q
\left(
\frac QT\right)^{2}
,
\eeann
where we used \eqref{eq:Q1bnd}. This completes the proof.
\epf

\newpage

\section{Proof of  Theorem \ref{thm:1}}\label{sec:thm1}

The purpose of this section is to
combine the ingredients in \S\S\ref{sec:vol}--\ref{sec:cE} to prove Theorem \ref{thm:1}.
Recall that
$$
\cN_{Q}(\tfrac{\xi}{ V_\G})=\frac{
%|\G_i|
1
}{2}\sum_{M\in\G
%/\G_i
\atop M\not\in K}\#\G\cap\cR_{M}(Q,\xi).
$$
%with $\Gamma_i=\Gamma\cap K$.
For a parameter $T$ to be chosen later, we use Proposition \ref{prop:cE} to write
$$
\cN_{Q}(\tfrac{\xi}{V_\G})=\frac{
%|\G_i|
1
}{2}\sum_{M\in\G%/\G_i
\atop M\not\in K,\|M\|<T}\#\G\cap\cR_{M}(Q,\xi) + O_{\xi}\left(Q^{2}{\log Q\over T^{2}}\right).
$$
Applying Proposition \ref{prop:cM} gives
$$
\cN_{Q}(\tfrac{\xi}{V_\G})=\frac{
%|\G_i|
1}{2}\sum_{M\in\G%/\G_i
\atop M\not\in K,\|M\|<T}{\vol(\cR_{M}(Q,\xi))\over V_{\G}} + O_{\xi}\left(
T^{34/9}Q^{(17+2\gT)/9}
+
Q^{2}{\log Q\over T^{2}}\right).
$$
Next apply Proposition \ref{prop:vol} to obtain
\beann
\cN_{Q}(\tfrac{\xi}{V_\G})&=&
{Q^{2} %|\G_i|
\over 2V_{\G}}
\sum_{M\in\G
%/\G_i
\atop M\not\in K,\|M\|<T}
\int_{0}^{\xi}f_{\gz}(\ell(M))d\gz
 + O_{\xi}\left(
T^{34/9}Q^{(17+2\gT)/9}
+
Q^{2}{\log Q\over T^{2}}\right)\\
&=&{Q^{2} \over 2V_{\G}}
\sum_{M\in\G\atop \|M\|<T}
\int_{0}^{\xi}f_{\gz}(\ell(M))d\gz
 + O_{\xi}\left(
T^{34/9}Q^{(17+2\gT)/9}
+
Q^{2}{\log Q\over T^{2}}\right).
\eeann
%using that $f_\xi(0)=0$ and that $\ell(M)=\ell(M')$ for any $M'\in M\Gamma_i$.
Here we have dropped the error term $O(T^{4}Q^{2/3})$ from Proposition \ref{prop:vol}, which will be of lower order.
Finally, use \eqref{eq:fxiDecay} to estimate
$$
\sum_{M\in\G\atop \|M\|<T}
\int_{0}^{\xi}f_{\gz}(\ell(M))d\gz
=
\sum_{M\in\G}
\int_{0}^{\xi}f_{\gz}(\ell(M))d\gz
+
O_{\xi}\left(T^{2}\frac1{T^{4}}\right)
,
$$
whence
$$
\cN_{Q}(\tfrac{\xi}{V_\G})
=
{Q^{2}\over 2 V_{\G}}
\sum_{M\in\G}
\int_{0}^{\xi}f_{\gz}(\ell(M))d\gz
 + O_{\xi}\left(
T^{34/9}Q^{(17+2\gT)/9}
+
Q^{2}{\log Q\over T^{2}}\right).
$$
Setting optimally
$$
T=
Q^{(1-2\gT)/52}
,
$$
we obtain
$$
\cN_{Q}(\tfrac{\xi}{V_\G})
=
{Q^{2}\over 2 V_{\G}}
\sum_{M\in\G}
\int_{0}^{\xi}f_{\gz}(\ell(M))d\gz
 + O_{\xi}\left(
Q^{2-\frac{(1-2\gT)}{26}+\vep}\right)
,
$$
thereby confirming \eqref{eq:cNQrate} with the rate claimed in \eqref{eq:cNQrateIs}. The rest of Theorem \ref{thm:1} follows immediately.

\newpage

\section{Proof of  Theorem \ref{thm:2}}\label{sec:thm2}
The purpose of this section is to prove that the pair correlation density $g_{2}(\xi)$ approaches $1$ in the limit as $\xi\to\infty$.
To do this we observe that $g_2(\tfrac{\xi}{V_\G})$ is a multiple of an automorphic kernel
$$g_2(\tfrac{\xi}{V_\G})= \frac{V_\G}{2\pi} K_\xi(1,1),$$
where
\begin{equation}\label{eq:cFxi}
K_\xi(g,h)=\sum_{M\in \G}f_\xi(\ell(gMh^{-1})).
\end{equation}
As the authors of \cite{BocaPopaZaharescu2013}
observed,
had the function $f_\xi$ been smooth, the Selberg (pre) trace formula would imply that
$$K_\xi(1,1)\sim\frac{1}{V_\G}\int_G f_\xi(\ell(g))dg,$$
and the result would follow by \eqref{eq:fxiInt}. The only difficulty is that the function $f_{\xi}$ in \eqref{eq:fXiDef} is not differentiable at two points, so one cannot apply the trace formula
 directly. Nevertheless, using a standard smoothing and unsmoothing argument (similar to the proof of Proposition \ref{prop:count}), we show the following

\begin{prop}\label{prop:smoothing}
As $\xi\to\infty$ we have
$$
K_\xi(1,1)=\frac{1}{V_\G}\int_G f_\xi(\ell(g))dg+O\left(\frac{1}{\xi^{(1-2\gT)/3}}\right).
$$
\end{prop}
\begin{proof}
As in the proof of Proposition \ref{prop:count},
let $\delta>0$ be a small parameter to be chosen later, and let $\delta_1$ be related to $\gd$ by \eqref{eq:gdTogd1}.
Let $\psi=\psi_\delta$ denote a spherical bump function supported on $B_{\delta_1}$, and
let
$$\Psi(g)=\sum_{\g\in \G}\psi(g\g),$$
where, as in \eqref{eq:L2norms}, we choose the bump function so that
$$\|\Psi\|\asymp \frac{1}{\sqrt{\vol(B_{\delta_1})}}\asymp \frac{1}{\delta}.$$

Note that for $g,h\in B_{\delta_1}$ we have $\ell(gMh^{-1})=\ell(M)+O(\delta)$ and using Lemma \ref{lem:fxivar} we see that %get
\beann
|K_\xi(g,h)-K_\xi(1,1)|&\leq & \sum_{M\in \G} |f_\xi(\ell(M))-f_{\xi}(\ell(gMh^{-1}))|\\
&\ll& \sum_{\|M\|\ll \sqrt{\xi}}\frac{\sqrt{\delta}}{\xi^2}+\sum_{\|M\|\ll \xi}\frac{\delta}{\xi^2}+\sum_{\|M\|\gg \xi}{\gd\over \|M\|^{4}}\\
&\ll&
\frac{\sqrt{\delta}}{\xi}+\delta+{\gd\over\xi^{2}}
\ll
\frac{\sqrt{\delta}}{\xi}+\delta%+{\gd\over\xi^{2}}
.
\eeann
Since $\psi$ is supported on $B_{\delta_1}$ and has integral one, %this implies
we have
that
\begin{equation}\label{eq:K1TKPsi}
K_\xi(1,1)=\<K_\xi,\Psi\otimes\Psi\>+%O(\delta)+
O\left(\frac{\sqrt{\delta}}{\xi}+\gd%+\frac1{\xi^{2}}
\right).
\end{equation}

On the other hand, unfolding as in \eqref{eq:unfold}, we obtain
\begin{equation*}
\<K_\xi,\Psi\otimes\Psi\>=\int_G f_\xi(\ell(g))\<\pi(g)\Psi,\Psi\>dg.
\end{equation*}
Decomposing $\Psi=\frac{1}{%\sqrt{
V_\G}+\Psi^{\bot}$ together with %decay of matrix coefficients
\eqref{eq:MtrxC}, namely that
$$\<\pi(g)\Psi^\bot,\Psi^\bot\>\ll \|\Psi\|^2\|g\|^{-1+2\gT}\ll \delta^{-2} \|g\|^{-1+2\gT},$$
gives
\begin{eqnarray}\nonumber
\<K_\xi,\Psi\otimes\Psi\>&=&\frac1{V_\G}{\int_G f_\xi(\ell(g))dg}+O\left(\delta^{-2}\int_G f_\xi(\ell(g))\|g\|^{-1+2\gT}dg\right)\\
\label{eq:KPsiEst} &=& \frac1{V_\G}{\int_G f_\xi(\ell(g))dg}+O\left(\frac{1}{\delta^{2}\xi^{1-2\gT}}\right)
\end{eqnarray}
where we used \eqref{eq:fxiBound}.

Finally, combining \eqref{eq:K1TKPsi} and \eqref{eq:KPsiEst}, with an optimal choice of
$$\delta={1\over\xi^{(1-2\gT)/3}},$$
gives
$$K_\xi(1,1)=\frac1{V_\G}{\int_G f_\xi(\ell(g))dg}+O(\xi^{(2\gT-1)/3}),$$
as claimed.
\end{proof}
Theorem \ref{thm:2} now follows immediately from Proposition \ref{prop:smoothing} together with \eqref{eq:fxiInt}.

%\section{Comparison with \cite{BocaPopaZaharescu2013}}
%Our result shows that (with the normalization $\vol(B_Q)\sim Q^2$) we have $g_2(\xi)\to \frac{2/\pi}{V_\G}$ as $\xi\to\infty$.
%In order to compare that with the Conjecture of \cite{BocaPopaZaharescu2013} we need to compare the two definitions.
%If we let $\tilde\cN_Q(\xi),\;\tilde R_2(\xi)$ and $\tilde g_2(\xi)$ denote the functions defined in \cite{BocaPopaZaharescu2013}.
%We recall that,
%$$\tilde\cN_Q(\xi)=\#\{\gamma,\gamma'\in B_Q: 0\leq \frac{1}{2\pi}(\theta_\gamma-\theta_\gamma')<\frac{\xi V_\G}{Q^2}\},$$
%so that
%$$\tilde \cN_Q(\xi)=\tfrac{1}{2}\cN_Q(2\pi V_\G\xi).$$
%We thus get that $\tilde{R_2}(\xi)=\tfrac{1}{2}R_2(2\pi V_\G\xi)$ and $\tilde{g}_2(\xi)=\pi V_\G g_2(2\pi V_\G\xi)$.
%With this comparison we see that as $\xi\to \infty$
%$$\lim_{\xi\to\infty}\tilde{g}_2(\xi)=\pi V_\G \lim_{\xi\to\infty} g_2(2\pi V_\G\xi)=2\neq 1$$

\newpage

%\bibliographystyle{alpha}

%\bibliography{AKbibliog}

\end{document}